\documentclass[11pt]{article}

\usepackage{amssymb,amsmath,amsfonts,amssymb,amsthm,mathrsfs}
\usepackage{graphics,graphicx,color,hhline}
\usepackage{hyperref}
\usepackage{bbm} 
\usepackage{authblk} 
\usepackage{euscript}
\usepackage{color}

\definecolor{darkgreen}{rgb}{0,0.5,0}
\definecolor{darkblue}{rgb}{0,0,0.7}
\definecolor{darkred}{rgb}{0.9,0.1,0.1}
\definecolor{lightblue}{rgb}{0,0.51,1}

\numberwithin{equation}{section}

\def\@abssec#1{\vspace{.05in}\footnotesize \parindent .2in 
{\bf #1. }\ignorespaces} 

\graphicspath{{/EPSF/}{Figures/}}   

\newtheorem{theorem}{Theorem}[section]
\newtheorem{lemma}[theorem]{Lemma}
\newtheorem{proposition}[theorem]{Proposition}

\theoremstyle{definition}

\def\ds{\displaystyle}
\def \Rm {\mathbb R}

\def \Cm {\mathbb C}

\newcommand{\eps}{\varepsilon}
\newcommand{\E}{\mathbb E}

\newcommand{\be}{\begin{equation}}
\newcommand{\ee}{\end{equation}}
\newcommand{\bea}{\begin{eqnarray}}
\newcommand{\eea}{\end{eqnarray}}
\newcommand{\bee}{\begin{eqnarray*}}
\newcommand{\eee}{\end{eqnarray*}}

\newcommand{\bP}{\mathbf P}

\def\fref#1{{\rm (\ref{#1})}}

\newcommand{\calL}{\mathcal L}

\newcommand{\calN}{\mathcal N}

\newcommand{\calW}{\mathcal W}

\newcommand{\cG}{\mathcal{G}}

\newcommand{\cout}[1]{}

\begin{document}

\author{Wenjia Jing \thanks{Yau Mathematical Sciences Center, Tsinghua University, Beijing 100084, China. Email: {\tt wjjing@math.tsinghua.edu.cn}} \and Olivier Pinaud \thanks{Department of Mathematics, Colorado State University, Fort Collins, CO 80525. Email: {\tt pinaud@math.colostate.edu}}}

\title{A backscattering model based on corrector theory of homogenization for the random Helmholtz equation}

\maketitle

\begin{abstract}

This work concerns the analysis of wave propagation in random media. Our medium of interest is sea ice, which is a composite of a pure ice background and randomly located inclusions of brine and air. From a pulse emitted by a source above the sea ice layer, the main objective of this work is to derive a model for the backscattered signal measured at the source/detector location. The problem is difficult in that, in the practical configuration we consider, the wave impinges on the layer with a non-normal incidence. Since the sea ice is seen by the pulse as an effective (homogenized) medium, the energy is specularly reflected and the backscattered signal vanishes in a first order approximation. What is measured at the detector consists therefore of corrections to leading order terms, and we focus in this work on the homogenization corrector. We describe the propagation by a random Helmholtz equation, and derive an expression of the corrector in this layered framework. We moreover obtain a transport model for quadratic quantities in the random wavefield in a high frequency limit.  
\end{abstract}

\section{Introduction}

This work is motivated by the study of electromagnetic wave propagation in sea ice. The latter is a formidably complex multiscale material, as a composite of pure ice, brine pockets, and air inclusions of different sizes, shapes, and contrasts. Sea ice is often represented as a layered medium, where the volume fraction and the nature of the inclusions vary from layer to layer. Sea ice is dispersive, mostly because of the high salt content of brine, and anisotropic due to the particular elongated shape of the brine pockets. Its physical properties, e.g. complex permittivity, depend on various parameters such as the temperature and the salinity, and there is quite a large literature on experimental estimations of such quantities, see e.g. the monograph \cite{bookice}. From a theoretical viewpoint, sea ice is modeled by a background (the pure ice), with randomly located inclusions of brine and air with shapes following some appropriate statistical distributions. There is usually a distinction between the young ice that is a few meters thick, and the multi-year ice that can be up to ten meters thick \cite{holt}.

The problem we are interested in here is the estimation, from radar data, of the thickness of the sea ice layer together with its effective permittivity. The particular experimental configuration we consider is that of a plane,  equipped with a radar system, flying at a given height and along a given path. The main question is then to determine whether the layer is sufficiently thick for the plane to land. The essential difficulty in doing so is that the beam emitted by the radar system is not perpendicular to the sea ice layer, but impinges on it with a given angle (the reason for this is to remove the so-called left-right ambiguity in some radar systems). 

This difficulty is explained as follows. In the frequency range the radar operates, say from 100MHz to 1GHz (higher frequencies do not penetrate well in the sea ice, see e.g. \cite{dierking}, and are therefore not adequate for our purpose), there is a clear separation of scales between the wavelength and the size of the inclusions: typical numbers given in \cite{vant} are 5mm for the air bubbles radius, and 5mm for the major axis of brines pockets with 0.025mm for the transversal axis; since the wavelength ranges from about 30cm to a few meters, it is then always significantly larger than the inclusions. This corresponds to a homogenization regime, where wave propagation in a highly heterogeneous medium is accurately approximated by wave propagation in a constant, effective medium. Note that even though the volume fraction of the brine pockets is not too large (10\% for brine pockets, 25\% for air bubbles in some appropriate conditions \cite{vant}), their effect is not negligible due to their very large contrast: while the dielectric constant of pure ice is about 3.14 with essentially no absorption, that of brine is about 70 with an absorption of 20 \cite{vant}. 

If the plane is at an altitude of 50 to 100 meters, we can assume that waves propagating in the air are in a high frequency regime, and therefore that ray theory applies. The consequence of this is, within the homogenization/ray theory picture and assuming the air/sea ice interface is flat,  that there should simply be no backscattered signal at the antenna as an application of Snell-Descartes laws. This is of course not the case in practice, and means that what is measured at the detector corresponds to correctors to these approximations. Deriving models for these measurements is the main objective of this work.

There are various sources of corrections to the homogenization/ray theory description: (i) surface roughness at the air/sea ice interface (ii) surface roughness at the sea ice/seawater interface (iii) surface roughness between layers in the sea ice (iv) correctors to the geometrical optics approximation (v) correctors to the homogenization limit. Besides, a complete description of the electromagnetic wave propagation in the sea ice involves the resolution of Maxwell's equations, in particular in order to account for anisotropy and polarization effects. Our ultimate goal is then to treat (i)--(v) in Maxwell's picture to model the backscattered signal. This is an extremely challenging task to account for all these phenomena at once, all the more in the context of Maxwell's equations. We decided therefore to start our program by designing a simplified, yet realistic, model for wave propagation that retains the essential feature of the true physical situation, namely that the measured data are essentially obtained from corrections to the homogenization/ray theory approximation.

We will then describe the propagation by a scalar wave equation, namely the Helmholtz equation, ignoring anisotropy and polarization. We will also suppose that we are in a regime where the effective (homogenized) coefficients of sea ice are simply given by the expectation of the random coefficients. This allows us to obtain a relatively direct characterization of the corrector to homogenization as a Gaussian field. We will provide analytical conditions for this assumption to hold, and also identify experimental situations where the assumption seems reasonable. Note that when the homogenized coefficients are not given by the expectation, the characterization of the corrector is considerably more difficult. There has been a lot of progress in this direction over the last years, mostly in the context of elliptic (diffusion) equations, see \cite{arms-ecole, gloria-jems, gloria-inv} for a few references. One of our future research plans is then to adapt these new results to our present problem. Finally, the characterization of the corrector in the stochastic homogenization of Maxwell's equations is an even more difficult question; to the best of our knowledge a precise characterization of the corrector in that setting remains open to this day.

Two types of asymptotic analysis will be performed in this work: first, a fluctuation theory to characterize the limit of the corrector to the homogenized model; then a high frequency limit via Wigner transforms to obtain a simplified transport model for correlations of the wavefield. We will provide conditions for these two limits to hold at once. In the mathematical analysis, we will focus on the corrector to homogenization aspects by assuming that the sea ice consists of only one layer with flat interfaces. We will then generalize our results without proofs to a multi-layer non-flat interfaces situation assuming there is a separation of scales between the sea ice inhomogeneities and the rough boundaries. Our proofs are based on adaptations of fairly standard methods, and what seems original in this work is their combination and the application to the sea ice problem.

The paper is structured as follows: In section \ref{secmod}, we set up the model, define the scalings and the random field modeling the heterogeneities. In section \ref{results}, we state our main results: the characterization of the corrector to homogenization is given in Theorems \ref{thm:homog} and \ref{thm:clt}, and a transport model for the correlations of the corrector in Theorem \ref{th:HF}. In section \ref{secgen}, we provide generalizations to more complex settings of multi-layer ices and non-flat interfaces. A conclusion is offered in section \ref{conc}. The other sections of the article are dedicated to the proofs of our main results, in particular to estimates on the Green's function of the homogenized problem.

\paragraph{Acknowledgments.} WJ's work is partially supported by the NSF of China under Grant No.\,11701314. OP's work is supported by NSF CAREER Grant DMS-1452349. OP is indebted to Professor Margaret Cheney for bringing up the problem addressed in this work and for many stimulating discussions.

\section{The Mathematical Model} \label{secmod}
Our starting point is the Helmholtz equation in $\Rm^d$, with $d=2,3$:
\be \label{helm}
\Delta u+\frac{\omega^2}{c^2(x)}u =S_\omega, \qquad x \in \Rm^d.
\ee
The function $c(x)$ is the speed, and $S_\omega$ is the source with (angular) frequency $\omega$. A point $x \in \Rm^d$ will be written as $x=(x',z)$, where $z \in \Rm$ and $x' \in \Rm^{d-1}$. We assume that the medium of propagation consists of three layers separated by two horizontal flat interfaces: the upper half-space (the air) $\{(x',z) \in \Rm^d \,:\, z > 0\}$, the sea ice $\{-L_0 < z < 0\}$, and the seawater $\{z < -L_0\}$ below the sea ice. The speed function $c(x)$ hence has the form
\begin{equation*}
\frac{1}{c^2(x)}=
\begin{cases}
 \frac{1}{c_0^2} \qquad &z>0,\\
\frac{1}{c_1^2}+V(\frac{x}{\ell_c}) \qquad &z \in (-L_0,0),\\
\frac{1}{c_2^2} \qquad &z<-L_0.
\end{cases}
\end{equation*}
Here, $c_0$ is the free space velocity, $c_1=c_0/n_1$, where $n_1$ is the (complex) refractive index of pure ice; similarly, $c_2=c_0/n_s$, and $n_s$ the refractive index of the sea water.  The term $V$ is a complex-valued random field that models heterogeneities at the scale $\ell_c$, and we assume it is such that the real part of $1/c^2(x)$ is strictly positive. We will give a concrete example for $V$ later towards the end of the section. 

Equation \fref{helm} needs to be supplemented with appropriate conditions at the infinity in order to have a unique solution. In the layered (waveguide) structure we consider here, it is not fully straightforward to adapt the classical Sommerfeld radiation conditions. The matter is discussed in \cite{ciraolo1} in two dimensions, and the correct conditions are found by studying separately the different kind of modes supported by the problem (e.g. evanescent or radiative). In order to simplify the mathematical analysis, we bypass the question of the boundary conditions and add some (small) artificial absorption to \fref{helm}, which immediately yields a unique solution in $L^2(\Rm^d)$ provided that $S_\omega \in L^2(\Rm^d)$. This has essentially no effects on our results since our estimates are made independent of the artificial absorption. The equivalence between imposing Sommerfeld conditions and adding small absorption is not trivial to establish and is investigated in \cite{castella-rad} for instance.

When the radar system is carried by a plane flying at height $H$, with velocity $V_p$ along a straight line in the direction of $x_1$ axis, $S_\omega$ has the form (for $d=3$),
\be \label{source}
S_\omega(x)=g((\omega-\omega_0)B^{-1}) \chi_0(x_1-V_p t,x_2,z-H), \qquad x=(x_1,x_2,z). 
\ee
Since $V_p$ is usually much smaller than $c_0$, the source is considered as fixed. Letting $t$ vary yields different sets of sources and measurements, which will be useful when calculating statistical averages provided that the underlying random medium is stationary and ergodic. Above, $g$ models the frequency distribution of the source with central frequency $\omega_0$ and bandwidth $B$.  The function $\chi_0$ is typically the characteristic function of a rectangle appropriately oriented (this models a beam with a direction orthogonal to the rectangle).

The simple model we just set up can be readily generalized as follows: firstly, while we considered above only one type of heterogeneities at the scale $\ell_c$ (for instance the brine pockets), additional (independent) random fields could be added to account for more inclusions (e.g. the air bubbles); secondly, a multi-layer model in the sea ice could be considered, to separate for instance young ice to multi-year ice. We will perform the mathematical analysis in the simplest case, and state the generalized results without proofs since the augmented models do not require further essential modifications except for more involved algebra. Finally, we remark on the situation when the interfaces are not flat and vary at a scale larger than that of media heterogeneities. From this separation of the scales, we expect that our results hold with just a rewriting of the interfaces. Nevertheless, our method of proof would have to be adapted since it relies on the translational invariance in the transverse variables, which would no longer hold. 

\medskip

The next section is devoted to the comparison between various length scales of the problem and to the discussion of appropriate scalings.

\subsection{Scalings} The largest scale in the problem is the height of plane $H$ (tenths of meters), and the smallest one is the typical size $\ell_c$ of the heterogeneities; we recall that the latter is around the millimeter for the air bubbles and for the long axis of the brine pockets. The central wavelength is
$$
\lambda_0:=\frac{2 \pi c_0}{ \omega_0}.
$$
In the regime we consider here, the central frequency is between 100MHz and 1Ghz, which corresponds to wavelengths between roughly 3m and 30cm. As a consequence, 
$$
\lambda_0 \gg \ell_c.
$$
Since the wavelength in pure ice is typically between $\lambda_0/2$ and $\lambda_0$, this means that the length scale of heterogeneity is much smaller than the wavelength, and this is the typical homogenization regime: the wavefield in the heterogeneous sea ice can be well approximated by waves in a homogeneous effective medium. Deriving the homogenized equation and the first-order corrector is the first step of our analysis. We assume that the above relation between the length scales holds for all frequencies in the bandwidth of the source. In other words, we impose that, with $
\omega=\frac{2 \pi c_0}{ \lambda},$
$$
\frac{\lambda_0}{\ell_c}\gg \frac{|\lambda-\lambda_0|}{\ell_c}, \qquad \textrm{so that} \qquad \frac{\lambda}{\ell_c}=\frac{\lambda_0}{\ell_c}+\frac{\lambda-\lambda_0}{\ell_c} \simeq \frac{\lambda_0}{\ell_c} \gg 1.
$$
In terms of $B$, this means $B \ll \omega_0$. 

\medskip

In the second step of the analysis, we derive transport equations for the signals scattered by the sea ice. They are the ones associated with the corrector term. We will see that the latter is a Gaussian field, and, as a consequence, the information is contained in averages of quadratic quantities of the field which can be asymptotically described by transport equations. This corresponds to a high frequency regime. With $H$ of order of tenths of meters, and the central wavelength $\lambda_0$ between 3m and 30cm, we clearly have $H \gg \lambda_0$ and the high frequency assumption holds in the air. In the sea ice layer, of thickness between 5m and 10m (these are numbers associated with multi-year ice, younger ice is typically less than 2m thick and does not seem able to withstand the weight of a plane and is therefore not considered), the wavelength in the effective medium is around $\lambda_e=\lambda_0/2$ for a source with range 100Mhz-1Ghz. Hence, for frequencies larger than 300Mhz, the ratio $\lambda_e/L_0$ is less than 0.1 (this is the worst case), and we can assume as well that the high frequency assumption holds. We then define the following non-dimensional parameters
$$\eta=\frac{\lambda_0}{H} \ll 1, \qquad \eps=\frac{\ell_c}{\lambda_0} \ll 1,$$
and we set
$$
\beta=\eps \eta \ll 1.   
$$
We rescale the spatial variable as $x =x' H $ (where $x'$ is the new variable though, in the presentation, we drop the primes). With $k = 2\pi H/\lambda_0 =2 \pi /\eta \gg 1$ the rescaled wavenumber, $L=L_0 /H$, and $f(x)= H^2 S_\omega( H x)$, \fref{helm} becomes
\be \label{Hr}
\Delta u^\beta + k^2 \underline n^2_\beta(x) u^\beta=f,
\ee
where the complex refractive index $\underline n_\beta(x)$ reads 
$$
  \underline n^2_\beta(x)=\left\{
    \begin{array}{ll}
      \underline n^2_0(x)=n_0^2+i \alpha & z >0\\
      \underline n^2_{1,\beta}(x)=n^2_1 + i (\kappa_1+\alpha)+V(\frac{x}{\beta})& z \in (-L,0)\\
      \underline n^2_2(x) =n_2^2+ i (\kappa_2+\alpha) & z<-L.
    \end{array}
    \right.
    $$
Above, $\alpha \ll 1$ is the artificial absorption, $n_0=1$, $n_1$ (resp. $n_2$) is the real part of the index of pure ice (resp. seawater), and $\kappa_1$ and $\kappa_2$ the absorption in pure ice and in seawater. Note that $(n_j,\kappa_j)$, $j=\{1,2\}$, depends on the frequency even though we will not make this explicit. We chose the above representation of $\underline n_\beta$ for simplicity as it allows us to factor out the term $k^2$. The absorptions will be rescaled by a factor $k$ later on. The parameter $L$ is of order $0.1$ if $H \sim 50$m and $L_0 \sim 5$m, which we consider to be large compared to $\eta$ of order 0.01 for $\lambda_0 \sim 50$cm, and to $\eps$ of order $0.005$ for $\ell_c \sim 1$mm. We will therefore keep $L$ fixed in the expansions.

\paragraph{Some notations.} We denote by $S=\{ x=(x',z)\in \Rm^d: z \in (-L,0)\}$ the (rescaled) sea ice layer, and will use the principal square root of a complex number, defined by, for $a=u+i v$ with $v \neq 0$,
\begin{equation*}
\sqrt{a}=\frac{1}{ \sqrt{2}} \left( \sqrt{\sqrt{u^2+v^2}+u}+i\,\mathrm{sgn}(v) \sqrt{\sqrt{u^2+v^2}-u} \right).
\end{equation*}
The complex conjugate of a number $a \in \Cm$ is denoted by $a^*$; the real and imaginary parts of $a$ is written as $\Re a$ and $\Im a$ respectively. The statistical average of a random variable $X$ is denoted by $\overline{X}$.
\medskip

 The next section is dedicated to the construction of the random field $V$.

\subsection{The random field}

The random field $V\equiv V(x,\theta)$ is defined on some probability space $(\Omega, \mathscr{F}, \mathbb{P})$, and we will omit the variable $\theta \in \Omega$ throughout the paper for simplicity when there is no possible confusion. We decompose $V$ into
$$V\left(\frac{x}{\beta},\theta\right)=\E\left\{V\left(\frac{x}{\beta},\theta\right)\right\}+\sigma q\left(\frac{x}{\beta},\theta\right),
$$
and denote by $q^\beta$ the rescaled field $q(\cdot/\beta)$, with $\beta$ defined before, and $q_r$ and $q_i$ the real and imaginary parts of $q$. Above, $\sigma$ is real and $\sigma^2$ is the variance of $V$ (and therefore $q$ has variance one), which does not depend on $x$ by the stationarity hypothesis below. We make the following assumptions on $V$:

\begin{enumerate}
\item[(A1)] $V$ is stationary with mean $n_V^2+ i \kappa_V$, with $\kappa_V>0$. This means that there exists a random variable $\tilde V$ and a family of measure-preserving translations $\{\tau_{x}\}_{x\in \Rm^d}$, such that $\E \{\tilde V\} = n_V^2+ i \kappa_V$ and $V(x,\theta) = \tilde{V}(\tau_x \theta)$.
\item[(A2)] There exists $C_0 > 0$ deterministic, such that $\|q\|_{L_x^\infty} \le C_0$ a.e. in $\Omega$.
\item[(A3)] There exists $n_m$ and $\kappa_m$ such that $\Re \{\underline n_\beta^2(x) \}\geq n_m^2>0$ and $\Im \{\underline n_\beta^2(x)\} \geq \kappa_m>0$, for all $x \in S$.
\end{enumerate}

The boundedness assumption (A2) is made to avoid technicalities and can be weakened; assumption (A3) is there to ensure that $\underline n^2_{\beta}(x)$ makes sense from a physical standpoint for all realizations. For the model \eqref{Hr}, the effective (i.e. homogenized) coefficient is simply given by the expectation of the random coefficient, that is, using the stationarity of $V$,
\be \label{defnH}
\underline n^2(x):=\E\{ \underline n_\beta^2\}=\left\{
    \begin{array}{ll}
       \underline n^2_0(x)=n_0^2+i \alpha & z >0\\
      \underline n^2_e(x)=n^2_e + i (\kappa_e +\alpha)& z \in (-L,0)\\
      \underline n^2_2(x) =n_2^2+ i (\kappa_2+\alpha) & z<-L,
    \end{array}
    \right.
    \ee
where $n_e^2=n_1^2+n_V^2$ and $\kappa_e=\kappa_1+\kappa_V$. Later in this section, we will compare certain properties of this homogenized model with experimental data in the physics literature, and provide a more quantitative criterion for the validity of a model in which the homogenized coefficients are just the averages of the random coefficients.

We further assume that the random field $q$ gets decorrelated fast enough. To describe this, we introduce the maximal correlation function $\varrho: (0,\infty) \to [0,1]$, defined to be
\begin{equation*}
\varrho(r) := \sup \left\{\frac{|\E gh|}{\sqrt{\E g^2 \E h^2}} \,\,\Big|\,\, g\in L^2_0(\mathscr{F}_U), h \in L^2_0(\mathscr{F}_V), \mathrm{dist}(U,V) \ge r, U\in \mathcal{C}, V \in \mathcal{C}\right\},
\end{equation*}
where $\mathcal{C}$ denotes the space of compact sets in $\Rm^d$, $\mathrm{dist}(U,V) = \min_{x\in U, y\in V} |x-y|$ is the distance from $U$ to $V$, $\mathscr{F}_U$ denotes the sub-$\sigma$-algebra of $\mathscr{F}$ generated by $q_r\rvert_U$ and $q_i\rvert_U$, and $L^2_0(\mathscr{F}_U)$ is the space of square integrable mean zero real random variables measurable with respect to $\mathscr{F}_U$. We assume that:
\begin{enumerate}
\item[(A4)] the maximal correlation function satisfies
\begin{equation*}
\int_{\Rm_+} \sqrt{\varrho(r)} f_d(r) dr < \infty,
\end{equation*}
with $f_3(r)=r^2$ and $f_2(r)=r(1+|\log r|)$.
\end{enumerate}

\noindent In addition, we will need the following matrix $M$:
\be \label{defM}
M=\left(
\begin{array}{ll}
\sigma_r^2 & \gamma\\
\gamma & \sigma^2_{i}
\end{array}
\right)
\ee
where $$
\sigma_r^2= \int_{\Rm^d} \E \{q_r(0) q_r(x)\} dx, \quad \sigma_i^2= \int_{\Rm^d} \E \{q_i(0) q_i(x)\} dx, \quad \gamma = \int_{\Rm^d} \E \{q_r(0) q_i(x)\} dx.
$$
 Note that $M$ is nonnegative definite by Bochner's theorem, and that the latter integrals are all finite as a consequence of (A4), as for instance,
\be \label{quad}
|\E \{q_i(0) q_i(x)\} | \leq \varrho(|x|) \E \{q_i(0)^2\}.
\ee
The first two integrals are positive according to Bochner's theorem. 

\paragraph{A  concrete model for $V$.} A natural way to model the distribution of the heterogeneities is by a collection of spheres with radii $\{R_i\}$ centered at $\{x_i\}$ and with (complex-valued) contrasts $\{\tau_i \}$. We then consider
$$
W(x)=\sum_{i=1}^\infty  \tau_i h\left(\frac{x-x_i}{R_i} \right),
$$
for $h$ the characteristic function of the unit sphere. Note that the elongated shape of the brine pockets could easily be implemented in this framework by setting for instance

$$
W(x)=\sum_{i=1}^\infty  \tau_i h\left(\frac{\tilde x'-\tilde x'_i}{\gamma R_i}, \frac{\tilde z-\tilde z_i}{R_i}, \right),
$$
where $R_i$ is the length of the pockets along the principal axis, $\gamma<1$ models the aspect ratio, and the $\;\tilde{}\;$ in the coordinates represents a (random) rotation of the axes modeling the orientation of the pockets.
  
The locations of the centers are random and are assumed to form a Poisson point process. This random distribution of spatial points satisfies the following: for any bounded, open (or closed) set $D \subseteq \Rm^d$, the number of points both in the random collection and in $D$ is a random number $N(B)$ that follows the Poisson distribution with mean $\lambda|B|$, where $|B|$ is the volume of the set $B$ and $\lambda$ is called the intensity of the Poisson point process. Also, given $N(B)$, the points $\{x_i\} \cap B$ are independent and uniformly distributed in $B$. We may also assume that the radii $\{R_i\}$ and contrasts $\{\tau_i \}$ are random and are i.i.d random variables with appropriate distributions. Since the inhomogeneities have a maximal and a minimal radius, we can assume that  $R_m \leq R_i \leq R_M$. In the same way, we have $0<\Re(\tau_m)\leq \Re(\tau_i) \leq \Re(\tau_M)$ for the air bubbles and the brine pockets, with a similar relation for the imaginary parts. A simple calculation shows that
$$
\overline{W}:=\E \{W(x)\}=  \lambda\, \overline{\tau}\, \overline{v}
$$
where $\overline{\tau}$ is the average contrast and $\overline{v}$ the average volume of the heterogeneities. In the same way, the variance of $W$ is
$$
\textrm{Var}(W)=\E \left\{ \left(W(x)- \overline{W}\right)\left(W(x)- \overline{W}\right)^*\right\}= \lambda \overline{|\tau|^2} \overline{v}. 
$$
It is not expected that the constitutive parameters of one type of heterogeneities vary too much from one another, i.e. $\overline{|\tau|^2} \sim |\overline{\tau}|^2$, and as a consequence $(\textrm{Var}(W))^{1/2} \sim \overline{W}/ (\lambda \overline{v})^{1/2}$. Since, as mentioned in the introduction, the volume fraction of the inclusions is at least $10\%$,  then $(\lambda \overline{v})^{1/2} \geq 0.3$, and this shows that $\overline{W}$ and $(\textrm{Var}(W))^{1/2}$ are of the same order. 

This fact can help us investigate, in experimental settings, the validity of the assumption that the homogenized coefficients are just the averages of the random coefficients. We will see in the next section that the assumption holds provided that the wavelength is sufficiently large, and that the random medium fluctuations are not too strong. In the data provided in \cite[figure 4]{margaret-ice}, the real part of a diagonal component of the effective permittivity tensor ranges, depending on the frequency, from about 3.5 to 4.5  for a temperature of -14 degrees Celsius. More precisely, it is about 3.75 in the range 800Mhz-1Ghz. This means that in such configurations, the effects of the inclusions are not negligible, but are not too strong either, and correct the background of pure ice (with permittivity 3.14) by about 20\%. Based on the fact that $(\textrm{Var}(W))^{1/2} \sim \overline{W}$, it seems therefore reasonable to assume that the fluctuations around the expected value are not too large as well, and that the homogenized coefficients are well approximated by the expected value of the random permittivity. At higher temperatures, the experimental permittivity can increase up to 5.5, in which case the effective coefficients should not be obtained by a simple averaging, but rather by a much more involved analysis \cite{margaret-ice}.

We recover the scaling $V(x/\beta)$ as follows: set $\ell_c=R_M$, and suppose the intensity $\lambda$ satisfies $\lambda=\lambda_0 \beta^{-d}$. Then, setting $x=Hx$  in $W(x)$ yields, with the change of variables $R_i=r_i R_M $,
$$
V\left(\frac{x}{\beta} \right):=W(Hx)=\sum_{i=1}^\infty  \tau_i h\left(\frac{x}{\beta r_i} - x_i \right).
$$
Above, we used the fact that a Poisson point process $\{x_i\}$ with intensity $\lambda$ is statistically equivalent to a Poisson point process $\{ \beta x_i\}$ with intensity $\lambda \beta^{-d}$. The standard properties of Poisson processes show that $V$ is stationary and has finite range correlations, so that (A1) and (A4) are satisfied. Assumption (A3) is also trivially verified since the real and imaginary parts of $\tau_i$ are positive. Regarding (A2), even though for each realization the number of points $x_i$ in a bounded domain is finite, Poisson point processes allow clustering. This implies that the constant $C_0$ in assumption (A2) is not uniform in the randomness and that (A2) does not hold as stated. This issue can be fixed with additional technicalities that we just sketch here: replace the constant $C_0$ by $C_0 f(\beta)$, with $f(\beta) \to \infty$ as $\beta \to 0$. Consider then the event $\Omega_\beta$ that $|V(x/\beta)|$ is greater than $C_0 f(\beta)$. Its probability can be estimated by considering for instance the probability to find $N>f(\beta)$ points in a domain of size $\beta^d$, and can be shown to be very small as $\beta \to 0$. Define then $\tilde V=V$ on the complementary of $\Omega_\beta$ in $\Omega$, and $\tilde V=0$ on $\Omega_\beta$. Choosing an appropriate $f(\beta)$ and using $\tilde V$ to model the heterogeneities, our proofs can then be directly adapted to recover the corrector result of Theorem \ref{thm:homog}.

We present our main results in the next section.

\section{Main Results} \label{results}

\subsection{The homogenization and corrector results}
We suppose that the source $f$ is a smooth function with bounded support. In the regime we consider here, the solution $u^\beta$ to \fref{Hr}  converges as $\beta \to 0$ (in proper sense) to the solution $u$ of the following homogenized equation
\begin{equation}
\label{eq:hpde-s1}
\Delta u + k^2\underline{n}^2(x) u = f,
\end{equation} 
where $\underline n$ is defined in \fref{defnH}. We write $u= \cG f$ for the solution to \fref{eq:hpde-s1}, and have $u \in L^2(\Rm^d)$ thanks to the artificial absorption $\alpha$. We are mostly interested in the corrector $u^\beta-u$, which verifies 
\begin{equation}
\label{eq:exp01}
\Delta(u^\beta - u) + k^2\underline n^2(x)(u^\beta - u)= -k_\sigma^2 \chi q^\beta(x,\theta) u^\beta(x,\theta),
\end{equation}
where the randomness appears in the source term. Above, $k_\sigma^2 = \sigma k^2$, and $\chi$ is the characteristic function of the sea ice layer $S$.  To study the higher order correctors in $u^\beta - u$, we rewrite \eqref{eq:exp01} in a form that serves as the starting point of an iteration scheme:
\begin{equation*}
u^\beta - u = -\cG k_\sigma^2 \chi q^\beta u - \cG k_\sigma^2 \chi q^\beta (u^\beta - u).
\end{equation*}
Performing another step of iteration, we get
\begin{equation}
\label{eq:exp1}
u^\beta - u = -\cG k_\sigma^2 \chi q^\beta u + \cG k_\sigma^2 \chi q^\beta \cG k_\sigma^2 \chi q^\beta u +   \cG k_\sigma^2 \chi q^\beta  \cG k_\sigma^2 \chi q^\beta (u^\beta - u).
\end{equation} 
Defining $v^\beta=-\cG k_\sigma^2 \chi q^\beta u$, our first result below shows that, in terms of $\beta$, $u^\beta-u-v^\beta$  is of order $\beta^2$ for $d=3$ , and of order $\beta^2 |\log \beta|$ when $d=2$. We will see in our second result (Theorem \ref{thm:clt}) that $v^\beta$ is of order $\beta^{d/2}$, and as a consequence $v^\beta$ is the leading corrector as $\beta \to 0$.

\begin{theorem}\label{thm:homog}
Let $d=2,3$. Assume that the random field $q$ satisfies (A1)--(A4) and let $K$ be a bounded set in the upper half-space $\Rm_+^d=\{ x=(x',z) \in  \Rm^{d}: \; z>0\}$. Then, we have the estimate
\begin{equation}
\label{eq:corrector-s1}
\E \|u^\beta - u-v^\beta\| _{L^2(K)} \le C_{\beta,k,\sigma,\kappa_e} \beta^2 \|u\|_{L^2(S)},
\end{equation}
where
$$
C_{\beta,k,\sigma,\kappa_e}=\left\{
\begin{array}{ll}
C \sigma^2 k^{d/2} \kappa_e^{-1} \left(k^2+\beta k^d \kappa_e^{-1}\big(1+k^{-2} \kappa_e^{-1}(1+\sigma \kappa_m^{-1})\big)\right)& d=3\\
C \sigma^2 k^{d/2} \kappa_e^{-1} \left(k^2(1+|\log(k \beta)|)+k^d \kappa_e^{-1}\big(1+k^{-2} \kappa_e^{-1}(1+\sigma \kappa_m^{-1})\big)\right)& d=2,
\end{array}
\right.
$$
and $C$ is independent of $\beta$, $\alpha$, $k$, $\kappa_m$ and $\kappa_e$.
\end{theorem}

\medskip
In the theorem, the domain $K$ can be seen as the location of the detector where measurements are performed.  In \fref{eq:corrector-s1}, we kept track of the dependency of the constants with respect to $\sigma$, $k$, and the absorptions $\kappa_e$ and $\kappa_m$, in order to quantify how $\beta$ must relate to these parameters for the corrector result to hold. Relatively to the size of $u$ (measured here by $\|u\|_{L^2(S)})$, we need $C_{\beta,k,\sigma,\kappa_e} \beta^2\ll 1$, and remark first that for the wave to propagate in the sea ice layer, $\kappa_e$ and $\kappa_m$ have to be sufficiently small and of order $k^{-1}$. Indeed, when the absorption $\kappa_e$ is small compared to $n_e^2$ (this holds in most practical configurations, see \cite{margaret-ice}), the imaginary part of $\underline n_e$ behaves like $\kappa_e$, and we expect the wave energy to decrease exponentially with a factor proportional to $k \kappa_e |x|$. This means that $k \kappa_e$ has to be of order one for the wave not to go extinct at a shallow depth in the sea ice. We write then $\kappa_e \sim \kappa_m \sim k^{-1}$, and suppose that the variance of $V$ is not necessarily small, say $\sigma \geq 1$. The condition $C_{\beta,k,\sigma,\kappa_e} \beta^2\ll 1$ then becomes, to leading order,
$$
\beta^2 \sigma^2 k^{9/2} \ll 1 \quad \textrm{when} \quad d=3, \qquad \beta^2 \sigma^2 k^{4}(|\log (\beta k)|+\sigma k) \ll 1, \quad \textrm{when} \quad d=2.
$$
The second condition boils down to $\beta^2 \sigma^3 k^{5} \ll 1$ since when $\beta$ satisfies the latter condition, then the one involving $\log(\beta k)$ is in turn verified. Recalling that $\beta=\eps \eta$, and that $k=2\pi/\eta$, we find the conditions $\eps \sigma \ll \eta^{5/4}$ when $d=3$, and $\eps \sigma^{3/2} \ll \eta^{3/2}$ when $d=2$. These conditions hold when the wavelength is sufficiently large and the fluctuations not too strong.

The proof of Theorem \ref{thm:homog} is fairly standard and follows the lines of those e.g. of \cite{GB-08,Bal-Jing,figari}. It is based on estimates of fourth-order moments of $q$ and on estimates on the Green's function of \fref{eq:hpde-s1}. The main differences with these references are that we work in an unbounded domain, which brings in additional technical difficulties, and that we keep track of important constants, in particular of $k \gg 1$, in order to obtain a corrector result uniform in the main parameters in the problem.

The second result is a precise description of the limiting distribution of the  normalized homogenization error $(u^\beta - u)/\sqrt{\beta^d}$, in the limit as $\beta \to 0$ keeping the other parameters fixed. For its statement, we will need the square root of $M$, given by
$$
M^{1/2}=\frac{1}{t}\left(
\begin{array}{ll}
\sigma_r^2+s & \gamma \\
\gamma & \sigma^2_{i}+s
\end{array}
\right), \qquad s=\sqrt{\sigma_r^2 \sigma_i^2-\gamma^2}, \qquad t=\sqrt{\sigma_r^2+\sigma_s^2+2 s}.
$$
Denote by $\{\alpha_{ij}\}$ the entries of $M^{1/2}$, and let
\be \label{defWy}
 W_y=(\alpha_{11} W_y^{(1)} + \alpha_{12} W_y^{(2)})+i(\alpha_{21}W_y^{(1)}  + \alpha_{22} W_y^{(2)}) = \begin{pmatrix} 1 & i\end{pmatrix}M^{\frac12}\begin{pmatrix}W_y^{(1)} \\ W_y^{(2)}\end{pmatrix},
\ee
where $W^{(1)}_y$ and $W^{(2)}_y$  are standard independent multiparameter ($y$-parameter) Wiener processes \cite{khosh}. We have then the following result.

\begin{theorem}\label{thm:clt} Under the same assumptions as in Theorem \ref{thm:homog}, we have, for any bounded set $K$ in the upper half space,
\begin{equation*}
\frac{u^\beta - u}{\sqrt{\beta^d}} \xrightarrow[\beta\to 0]{\mathrm{distribution}} v(x)=-\sigma k^2 \int_{\Rm^d} G(x,y)\chi(y)u(y) dW_y \qquad \text{in} \qquad L^2(K),
\end{equation*}
where $G$ is the Green's function associated to problem \eqref{eq:hpde-s1}.
\end{theorem}
The result above is a convergence in distribution of functions in $L^2$ (Hilbert) spaces, see section \ref{secasymp} and \cite{partha} for an introduction on the subject. In Theorem \ref{thm:clt}, the Wiener integrals in the limiting corrector are simply, after integration in $x$ against a test function, normal random variable with zero mean and appropriate variances. 
The first-order corrector $v$ is therefore a mean-zero Gaussian field, whose information is contained in correlations of the form
$$
\E\{v(x) v^*(y)\}= k^4 \tau^2 \int_{\Rm^d} G(x,z) G^*(y,z) |\chi(z) u(z)|^2 dz, 
$$
with $\tau^2 = \sigma^2(\sigma_r^2+\sigma_i^2)$. Above, we have used that
\be \label{varWien}
 \E\{ d W_x d W^*_y\}= \mathrm{tr}(M)  \delta(x-y) = (\sigma_r^2 + \sigma_i^2) \delta(x-y). 
 \ee

In the next section, we approximate these correlations in the high frequency limit $\eta \to 0$, assuming the condition $C_{\beta,k,\sigma,\kappa_e} \beta^2\ll 1$ holds for the corrector result to be true.

\subsection{The high frequency limit} \label{sechigh}

With Theorems \ref{thm:homog} and \ref{thm:clt} at hand, we approximate the wavefield $u^\beta$ by $u^\beta_0=u+\beta^{d/2} v$, where $v$ is formally the solution to, recalling that $k=2\pi /\eta$,
\be \label{eqlimitv}
    \eta^2 \Delta v(x)+ (k^2(x)+i \eta \mu(x))  v(x)= -(2\pi)^2 \chi(x) u(x) dW_x,
    \ee
where we have introduced 
$$
k^2(x)=\left\{
    \begin{array}{ll}
      k_0^2=(2 \pi)^2  \underline n_0^2& z >0\\
      k_e^2=(2 \pi)^2\underline n^2_e & z \in (-L,0)\\
      k_2^2=(2 \pi)^2\underline n_2^2 & z<-L
    \end{array}
  \right. \quad
\mu(x)=\left\{
    \begin{array}{ll}
       \mu_0=(2 \pi)^2 \alpha^0& z >0\\
      \mu_e=(2 \pi)^2 (\kappa_e^0 +\alpha^0)& z \in (-L,0)\\
      \mu_2=(2 \pi)^2 (\kappa_2^0 +\alpha^0)& z<-L,
    \end{array}
    \right.
  $$
with $\alpha^0=\alpha/\eta$, $\kappa^0_e=\kappa_e/\eta$, $\kappa^0_2=\kappa_2/\eta$.

The main tool to study the high frequency limit $\eta \to 0$ of \fref{eqlimitv} is the Wigner transform, defined by, for a function $w$:
$$
\calW[w](x,p)=\frac{1}{(2 \pi)^d} \int_{\Rm^d} e^{i x \cdot p} w(x-\frac{\eta}{2} y)w^*(x+\frac{\eta}{2} y) dy.
$$
See \cite{LP,GMMP} for an introduction on Wigner transforms. Rigorous mathematical analyses with Wigner transforms are typically difficult and technically involved, all the more in domains with sharp interfaces, and are beyond the scope of this work. We hence decided to remain at a formal level in the present section. We refer to \cite{fouassier,miller-JMPA} for rigorous derivations.

The correlation $\E\{u_0^\beta(x) (u_0^\beta)^*(y)\}$ can be written in terms of $\calW[v]$: indeed, since $\E\{v\}=0$, we have 
    \be \label{wigcorrec}
\E \{ \calW[u^\beta_0]\}=\calW[u]+\beta^d \E \{\calW[v]\},
\ee
and an inverse Fourier transform yields

$$
\E\{u_0^\beta(x)(u_0 ^\beta)^*(y)\}=u(x)u^*(y)+\beta^d \int_{\Rm^d} e^{-i \frac{(x-y) \cdot p}{\eta}}\E\{\calW[v]\}\left(\frac{x+y}{2},p\right) dp.
$$
Our next result provides us with an asymptotic description of $\E \{\calW[v]\}$ as $\eta \to 0$. Before stating it, we need to introduce a few notations. For $|\xi| \leq k_e$, let $k_j(\xi)=(k_j^2-|\xi|^2)^{1/2}$, for $j=\{0,e,2\}$. We recall that $k_0<k_e<k_2$, and we set $k_0(\xi)=0$ when $|\xi|>k_0$. The reflection-transmission amplitudes at the interfaces $z=0$ (associated to $j=0$) and $z=-L$ (associated to $j=2$) are defined by

$$
R_j(\xi)=\left| \frac{k_e(\xi)-k_j(\xi)}{k_e(\xi)+k_j(\xi)} \right|^2, \qquad T_j(\xi)=\frac{4 k^2_e(\xi)}{(k_e(\xi)+k_j(\xi))^2}.
$$
The boundary values of a function $\calW(x,p)$ at these interfaces are denoted as follows: write $x=(x',z)$, $x' \in \Rm^{d-1}$, and $k=(k_\perp,k_z)$. Then,
$$
\calW_0^{\pm}(x',k_\perp,k_z):=\calW((x',0^\pm),(k_\perp,k_z)), \qquad \calW_L^{\pm}(x',k_\perp,k_z):=\calW((x',-L^\pm),(k_\perp,k_z)),
$$
where the upperscript $\pm $ refers to the upper/lower limit. We have then the following theorem.

    \begin{theorem} \label{th:HF} We have $\E\{\calW[v]\}=\calW+o(\eta)$, where $o(\eta)$ represents a term that converges to zero as $\eta \to 0$ in the distribution sense, and where $\calW$ satisfies
$$
( \mu(x)+ p \cdot \nabla_ x)\calW(x,p)=\frac{\pi (2\pi)^4 \tau^2}{\eta^{d+1}}\delta(|p|^2-k_e^2) |\chi u|^2(x),
$$
equipped with the following reflection-transmission conditions at the interface $z=0$,

$$
\begin{array}{ll}
\textrm{when} \qquad |k_\perp|<k_0 &
\qquad \left\{ \begin{array}{l}
\calW_0^{+}(x',k_\perp,k_0(k_\perp))=T_0(k_\perp) \calW_0^{-}(x',k_\perp,k_e(k_\perp))\\
\calW_0^{-}(x',k_\perp,-k_e(k_\perp))=R_0(k_\perp) \calW_0^{-}(x',k_\perp,k_e(k_\perp))
\end{array}\right.\\[4mm]
\textrm{when} \qquad |k_\perp|>k_0 &\qquad \calW_0^{-}(x',k_\perp,-k_e(k_\perp))=\calW_0^{-}(x',k_\perp,k_e(k_\perp)),
\end{array}
$$
and at $z=-L$,
\begin{align*}
&\calW_{L}^{+}(x',k_\perp,k_e(k_\perp))=R_2(k_\perp) \calW_L^{+}(x',k_\perp,-k_e(k_\perp))\\
&\calW_{L}^{-}(x',k_\perp,-k_2(k_\perp))=T_2(k_\perp) \calW_L^{+}(x',k_\perp,-k_e(k_\perp))
\end{align*}

      \end{theorem}

\bigskip 

The (formal) proof is direct and follows the lines of \cite{BR-AM}. The interpretation of Theorem \ref{th:HF} is the following: as the homogenized solution $u$ propagates in the sea ice layer, it interacts with heterogeneities at scales much smaller than the wavelength. This results in an isotropic radiation (modeled by the delta function $\delta(|p|^2-k_e^2)$) with intensity $|u(x)|^2$, which then propagates according to ray theory, and is refracted at the interface between the air and sea ice and the interface between sea ice and seawater following Snell-Descartes laws. Note that, even though the refractive index is complex in the sea ice and seawater, we still obtain the usual Snell-Descartes laws. This is because the absorption is small compared to the real part of the index. When this is not the case, Snell-Descartes laws have to me modified, see \cite{ray-complex}. Moreover, we do not need to consider the critical case $|k_\perp|=k_0$ in the boundary conditions since this situation corresponds to a set of measure zero in $\Rm^d$. There is no critical case at the bottom interface since the real part of the refractive index in seawater is larger than the effective index in sea ice, and therefore $k_2> k_e$.

Note that the corrector in \fref{wigcorrec} is of order $\beta^d \tau^2/ \eta^{d+1}$. Since $\tau \sim \sigma$, it is of order $\eps^d \sigma^2 / \eta$ in terms of $\eps$. Recalling that we are in a regime where $\sigma \geq 1$ and $\eps \sigma \ll \eta^{5/4}$ when $d=3$, and $\eps \sigma^{3/2} \ll \eta^{3/2}$ when $d=2$, we find $\eps^3 \sigma^2 / \eta \ll \sigma^{-1} \eta^{11/4}$ when $d=3$, and $\eps^2 \sigma^2 / \eta \ll \sigma^{-1} \eta^{2}$ when $d=2$. In both cases, the corrector is therefore small.

Using the stationarity of the random medium, the correlations can be calculated as follows. Recalling the form of the source term \fref{source}, and indexing $u^\beta$ as $u^\beta_t$, we have, invoking ergodicity, 
$$
\E\{u^\beta(x) (u^\beta)^*(y)\}=\lim_{T\to \infty} \frac{1}{T} \int_0^T u^\beta_t(x) (u^\beta_t)^*(y) dt.
$$
From a practical viewpoint, what is measured is 
$$
C(x,y):=\frac{1}{T_0} \int_0^{T_0} u^\beta_t(x) (u^\beta_t)^*(y) dt,
$$
for some $T_0$ sufficiently large and a few frequencies $\omega$ in the bandwidth. The correlation $C(x,y)$ is asymptotically modeled by 
$$
C_0(x,y):=u(x)u^*(y)+\beta^d \int_{\Rm^d} e^{-i \frac{(x-y) \cdot p}{\eta}}\calW\left(\frac{x+y}{2},p\right) dp.
$$
The next step is the resolution of an inverse problem with data $C(x,y)$ at the detector and model $C_0(x,y)$, with the goal of recovering some information about the sea ice,  namely the thickness of the layer and the homogenized coefficients. This will be addressed in future works.

The next section is devoted to generalizations of Theorems \ref{thm:homog} and \ref{thm:clt}.

\subsection{Generalizations} \label{secgen}

A multi-layer configuration can be considered as follows: suppose the sea ice consists of $N_L$ layers defined by
$$
L_i=\{ (x',z)\in \Rm^d: \; \ell_i(x')<z<\ell_{i+1}(x')\}, \qquad i=1,\cdots,N_L,
$$
and assume there are $N^i_H$ different types of heterogeneities in the layer $i$. We then define the new random field by
$$
V=\sum_{i=1}^{N_L} \sum_{j=1}^{N_H^i} \chi_i V_{ij},
$$
where the $V_{ij}$ are independent and have a typical scale much smaller than the wavelength.  We set $\ell_c$ to be the largest of these scales. Above, $\chi_i$ is the characteristic function of $L_i$. When the functions $\ell_i$ vary slowly compared to the scale $\beta=\ell_c/H$, i.e. $\max_{x' \in \Rm^{d-1}} |\ell'_i(x')| \ll \beta^{-1} $ in rescaled variables, then the situation is essentially similar to the flat interface case since the scales of variation of the surfaces and of the heterogeneities are well separated. It is then expected that, asymptotically in $\beta$,
\be \label{asgene}
u^\beta  \simeq u+ \beta^{d/2}v, \qquad v=-k^2 \sum_{i=1}^{N_L}\sum_{j=1}^{N^i_H} \int_{\Rm^d} G(x,y) \chi_i(y)u(y) dW^{(ij)}_y.
\ee
In the latter, $u$ and $G$ are the solution and the Green's function of the multi-layers multi-species homogenized Helmholtz equation, and $dW^{ij}_y$ has the form \fref{defWy}, where the $\alpha$ coefficients are the entries of square root of the matrix $M^{ij}$, which is defined in a similar way as in \fref{defM}.

Non-flat interfaces can also be considered; suppose that each interface profile consists of a slowly varying part, i.e.\,varying on scales much larger compared to the wavelength, augmented with a highly oscillatory (possibly random) part. When the oscillations are much faster than the wavelength, homogenization theory predicts a limit model where the rough boundary is replaced by the slow varying profile that is surrounded by a boundary layer where an equation with effective coefficients is solved. Outside of this layer, the original equation is satisfied with appropriate continuity conditions. See e.g. \cite{keller-boundary} for more details and the references therein. In terms of \fref{asgene}, this means that $u$ and $G$ are replaced by some homogenized versions $u_e$ and $G_e$. When the oscillations of the rough boundary are of the same order as the wavelength, the situation is more complicated. Nevertheless, if the amplitude of the fast fluctuations is sufficiently small, a transport model for $u(x)u^*(y)$ can be obtained provided the high frequency hypothesis holds in the layer of interest. See \cite{BPR-1} for more details. One of our next objectives is to compare numerically these various boundary effects with the volume scattering modeled the corrector $v$.

The rest of the paper is devoted to the proofs of our main results.

\section{Proofs of the homogenization and corrector results}
We start with some preliminaries about Green's functions. 
\subsection{The Green's function of the Helmholtz equation}

We derive some estimates for the Green's function $G(x,y)$ of the homogenized equation, i.e. the solution to
\begin{equation}
\label{eq:homG}
\Delta_x G(x,y) + k^2\underline n^2(x)G(x,y)  = \delta_y, \qquad x,y \in \Rm^d,
\end{equation}
where $\underline n$ is defined in \fref{defnH}, by exploiting the layered structure of the underlying medium and the physical absorption in the middle layer (sea ice). Note that the reciprocal property $G(x,y)=G(y,x)$ holds. Let $\Phi = \Phi(x,y)$ be the Green's function in free space, that is
\begin{equation}
\label{eq:FreeG}
\Delta_x \Phi + k^2 \underline n_e^2 \Phi = \delta_y, \qquad x,y \in \Rm^d,
\end{equation}
where $\underline n_e $ is the homogenized complex refractive index of sea ice, see \fref{defnH}. The function $\Phi$ admits the following explicit formulas:
\begin{equation*}
\begin{aligned}
&\Phi(x) =  \frac{e^{i k \underline n_e|x|}}{4 \pi|x|}, \qquad \textrm{for } d = 3,\\
&\Phi (x) = \frac{i}{4} H_0^{(1)}(k \underline n_e |x|), \quad \textrm{for } d = 2,
\end{aligned}
\end{equation*}
where $H_0^{(1)}$ is the zero order Hankel function of the first kind.  We have the proposition below.

\begin{proposition} \label{propG}Consider the Green's function $G(x,x_0)$ solution to \fref{eq:homG} and let $K$ be a bounded set in the upper half space $\Rm_+^d=\{ x=(x',z) \in  \Rm^{d}: \; z>0\}$. With \fref{defnH}, suppose that $n_e>n_0$ and that $n_2>n_0$. Set $p(x,x_0) = G(x,x_0)-\Phi(x-x_0)$ for $x_0 \in S$. Then there exists a generic constant $C$ independent of $k$ and $\kappa_e$ (and $\alpha$) such that
\be \label{plemma}
\sup_{x_0 \in S}\|p(\cdot,x_0)\|_{L^2(S)} \leq C k^{d/2-2} \kappa_e^{-1},
\ee
we have the following $L^\infty$ estimate
\be \label{plemma2}
\sup_{x_0 \in S}\|p(\cdot,x_0)\|_{L^\infty(S)} + \sup_{x_0 \in S}\|p(\cdot,x_0)\|_{L^\infty(K)}\leq C k^{d-2} \kappa_e^{-1},
\ee
and $G$ verifies
\be \label{Glemma}
\sup_{x_0 \in S}\|G(\cdot,x_0)\|_{L^2(S)} + \sup_{x_0 \in K}\|G(\cdot,x_0)\|_{L^2(S)} \leq C k^{d/2-2} \kappa_e^{-1}.
\ee
Moreover, for any $s < (4-d)/2$, there exists another constant $C$ that depends on $s$, $k$ and $\kappa_e$ such that
\be \label{estHsG}
\sup_{y \in S}\| G(\cdot,y)\|_{H^s(K)} \leq C,
\ee
where $H^s$ is the usual Sobolev space.

\end{proposition}

The proof is postponed to section \ref{secpropG}. When $d=2$, we will use the result below, proved in section \ref{proofG2d}:

\begin{lemma} \label{G2d}
  For $d=2$, we have the  estimate, 
  $$
  |\Phi(x)| \leq C +C|\log \big(k\Re(\underline n_e) |x|\big)|, \qquad \forall x \in \Rm^2,
  $$
where $C>0$ does not depend on $k$ and $\underline n_e$.
\end{lemma}
In the rest of the paper, we will use the notation
$$
R(x-y):=\E\{ q(x) q^*(y) \}.
$$
Note that, due to assumption (A4) and a similar relation to \fref{quad}, $R$ is integrable over $\Rm^d$.

\subsection{Error estimates}
The goal of this subsection is to prove Theorem \ref{thm:homog}. The strategy of proof is an adaptation of that of \cite{Bal-Jing}. The starting point is the relation.
\begin{equation}
\label{eq:exp11}
u^\beta - u = -\cG k_\sigma^2 \chi q^\beta u + \cG k_\sigma^2 \chi q^\beta \cG k_\sigma^2 \chi q^\beta u +   \cG k_\sigma^2 \chi q^\beta  \cG k_\sigma^2 \chi q^\beta (u^\beta - u).
\end{equation} 
With the goal of controlling the last term on the right above, we get a first direct estimate for $u^\beta - u$ with the lemma below. We recall that $\chi$ is the characteristic function of the sea ice domain $S$. 
\begin{lemma} \label{estimsol} Let $g \in L^2(S)$ and $v$ satisfy
\begin{equation}
\label{eq:helm}
(\Delta + k^2  \underline n_\beta^2(x)) v= \chi g
\end{equation}
with $\inf_{x \in S} \Im (\underline n^2_\beta(x))=\kappa_m >0$.
Then, we have the estimate
$$
\| v \|_{L^2(S)} \leq k^{-2} \kappa_m^{-1}\|g\|_{L^2(S)}.
$$
\end{lemma}
\begin{proof}
Multiplying \fref{eq:helm} by $v^*$, integrating over $\Rm^{d}$ and taking the imaginary part leads to
$$
k^2 \kappa_m \|v\|^2_{L^2(S)} \leq k^2 \int_{S} (\kappa_e+\alpha+ \sigma q_i(x/\beta))|v(x)|^2 dx= \Im \int_{S} g(x) v^*(x) dx,
$$
and it suffices to use the Cauchy-Schwarz inequality to conclude.
\end{proof}

\bigskip

With $v^\beta=-\mathcal{G} k_\sigma^2 \chi q^\beta u$, we then note the following:
\begin{equation*}
(\Delta + k^2  \underline n_\beta^2(x)(u^\beta - u - v^\beta) = -k_\sigma^2 \chi q^\beta v^\beta.
\end{equation*}
In view of Lemma \fref{estimsol} and assumption (A2), we find
$$
\|u^\beta-u-v^\beta\|_{L^2(S)} \leq \sigma \kappa_m^{-1} \| q^\beta v^\beta\|_{L^2(S)} \leq C_0 \sigma \kappa_m^{-1} \|v^\beta\|_{L^2(S)}.
$$
The next step is to estimate $v^\beta$. We use for this the following lemma, whose proof is postponed to the end of the section.

\begin{lemma}
\label{lem:cGL2}
Assume (A1)--(A4) hold. Let $\cG$ be the solution operator associated to problem \eqref{eq:hpde-s1}. Then, for any $h \in L^2(S)$, we have,
\begin{equation*}
\E \left\| \cG k^2_\sigma  \chi q^\beta h \right\|_{L^2(S)}^2 \le C \sigma^2 k^d \beta^d \kappa_e^{-2}\| h\|_{L^2(S)}^2 ,
\end{equation*}
where the constant $C$ is independent of $\alpha$, $\sigma$, $k$, $\beta$ and $\kappa_e$.
\end{lemma}

\bigskip 

With the shorthand $\| u\|^2_{L^2_\Omega(S)}= \E \|u\|^2_{L^2(S)}$, we then find
$$
\|u^\beta - u\|_{L^2_\Omega(S)} \le C(1+\sigma \kappa_m^{-1})\|v^\beta\|_{L^2_\Omega(S)} \leq C (1+\sigma \kappa_m^{-1}) \sigma (k\beta)^{d/2} \kappa_e^{-1} \|u\|_{L^2(S)},
$$
which yields the homogenization result provided $(1+\sigma \kappa_m^{-1}) \sigma (k\beta)^{d/2} \kappa_e^{-1} \ll 1$. Clearly, $\|u^\beta-u-v^\beta\|_{L^2(S)}$ satisfies the same bound; however, the same argument does not yield uniform in $\alpha$ control of $\|u^\beta - u -v^\beta\|_{L^2(K)}$, for a set $K$ in the upper half space, because the absorption in the upper space is precisely $\alpha$. We will use the above estimate together with the next lemma in order to control the higher order terms in \fref{eq:exp11}. Fourth order moments of $q^\beta$ are needed in the proof, which is given in section \ref{proofGL3}. Below, $\calL(E,F)$ denotes the space of bounded operators from the Banach space $E$ to the Banach space $F$.

\begin{lemma} \label{lem:cGL3} Assume (A1)--(A4) hold. Then, for any bounded set $K$ in the upper half-space,
\begin{equation}
\label{eq:Gopnorm1}
\E \|\cG k_\sigma^2 \chi q^\beta \cG \chi\|_{\mathcal{L}(L^2(S),L^2(K))}^2 \le  C^{(0)}_{k,\sigma,\kappa_e} \beta^d,
\end{equation}
and
\begin{equation}
\label{eq:Gopnorm2}
\E \|\cG k_\sigma^2 \chi q^\beta \cG k_\sigma^2 \chi q^\beta \|_{\mathcal{L}(L^2(S),L^2(K))}^2 \le \left\{
\begin{array}{ll}
C^{(1)}_{k,\sigma,\kappa_e} \beta^4(1+C^{(2)}_{k,\kappa_e}\beta^2)& \qquad d=3\\
C^{(1)}_{k,\sigma,\kappa_e} \beta^4\big((1+|\log(\beta k)|)^2+C^{(2)}_{k,\kappa_e}\big)& \qquad d=2,
\end{array}
\right.
\end{equation}
where $C^{(0)}_{k,\sigma,\kappa_e}=C \sigma^2 k^{2d-4} /\kappa_e^{4}$,  $C^{(1)}_{k,\sigma,\kappa_e}=C \sigma^{4}k^{d+4} /\kappa_e^{2}$, and $C^{(2)}_{k,\kappa_e}=k^{2d-4} /\kappa_e^{2}$.
\end{lemma}

\bigskip

We can now estimate the last term in \eqref{eq:exp11}. We have, with \fref{eq:Gopnorm1}, (A2), and the Cauchy-Schwarz inequality,
\begin{align*}
\E \|\cG k_\sigma^2 \chi q^\beta \cG &k_\sigma^2 \chi q^\beta (u^\beta - u)\|_{L^2(K)} \\
\le& k_\sigma^2 \|q^\beta\|_{L^\infty(\Omega \times \Rm^d)}\left( \E \|\cG k_\sigma^2 \chi q^\beta \cG\chi\|_{\mathcal{L}(L^2(S),L^2(K))}^2\, \E\|u^\beta - u\|_{L^2(S)}^2 \right)^{\frac 12}\\
\le& C^{(4)}_{k,\sigma,\kappa}\beta^d\|u\|_{L^2(S)},
\end{align*}
with $C^{(4)}_{k,\sigma,\kappa}=C (C^{0}_{k,\sigma,\kappa_e})^{1/2} (1+\sigma \kappa_m^{-1}) \sigma k^{d/2} \kappa_e^{-1}$. In terms of $\beta$, this term is of order $\beta^d$. 

Regarding the second term in the r.h.s of \fref{eq:exp11}, we find, with \fref{eq:Gopnorm2},
\begin{align*}
\E \|\cG k_\sigma^2 \chi q^\beta \cG k_\sigma^2 \chi q^\beta u\|_{L^2(K)} &\le \|u\|_{L^2(S)} \left(\E\|\cG k_\sigma^2 \chi q^\beta \cG k_\sigma^2 \chi q^\beta \|_{\mathcal{L}(L^2(S),L^2(K))}^2\right)^{\frac 12}\\
& \le   C^{(5)}_{\beta,k,\sigma,\kappa_e} \beta^2\|u\|_{L^2(S)},
\end{align*}
with 
$$
C^{(5)}_{\beta,k,\sigma,\kappa_e}=
\left\{
\begin{array}{ll}
\left(C^{(1)}_{k,\sigma,\kappa_e} (1+C^{(2)}_{k,\kappa_e}\beta^2)\right)^{1/2}& \qquad d=3\\
\left(C^{(1)}_{k,\sigma,\kappa_e} \big((1+|\log(\beta k)|)^2+C^{(2)}_{k,\kappa_e}\big)\right)^{1/2}& \qquad d=2,
\end{array} \right.
$$
This completes the proof of Theorem \ref{thm:homog} by inspection.

\subsection{Proof of Lemma \ref{lem:cGL2}} From direct calculations, we have
\begin{equation*}
\E \left\| \cG k_\sigma^2 \chi q^\beta h \right\|_{L^2(S)}^2 = k_\sigma^4 \int_{\Rm^{2d}\times S} G(x,y) G^*(x,y') \chi(y) \chi(y') R\left(\frac{y-y'}{\beta}\right) h(y)h^*(y') dy dy'dx.
\end{equation*}
Integrating over $x$ first and using Cauchy-Schwarz inequality, we find
\begin{equation*}
\E \left\| \cG k_\sigma^2 \chi q^\beta h \right\|_{L^2(S)}^2 \le k_\sigma^4 \sup_{y \in S}\|G(\cdot,y)\|^2_{L^2(S)} \int_{\Rm^{2d}} \left|R\left(\frac{y-y'}{\beta}\right)\right| \left| (\chi h)(y) (\chi h)(y')\right| dy dy'.
\end{equation*}
We can recast the remaining integral as
\begin{equation*}
\int_{\Rm^{d}} |\chi h|(y) \left( |R^\beta| *  | \chi h| \right) (y) dy dx,
\end{equation*}
where $R^\beta$ is a short-hand notation for $R(\cdot/\beta)$. Note that $R$ integrable, and $h$ is square integrable. We have then, using Young's inequality,
\begin{equation*}
\left\| |R^\beta| *  | \chi h| \right\|_{L^2} \le C \|R^\beta\|_{L^1} \| \chi h\|_{L^2} = C\beta^d \|R\|_{L^1} \| \chi h\|_{L^2},
\end{equation*}
and obtain finally
\begin{equation*}
\E \left\| \cG k_\sigma^2 \chi q^\beta h \right\|_{L^2(S)}^2 \le C \beta^d k_\sigma^4 \sup_{y \in S}\|G(\cdot,y)\|^2_{L^2(S)} \|R\|_{L^1} \|\chi h\|^2_{L^2}.
\end{equation*}
Using the first estimate of $G$ in \fref{Glemma}, we get the desired result and the proof is complete.

\subsection{Proof of Lemma \ref{lem:cGL3}} \label{proofGL3}  We start with the proof of \eqref{eq:Gopnorm1}. We have, after a Cauchy-Schwarz inequality, for all $x \in K$ and all $f \in L^2(S)$,
\begin{equation*}
\left\lvert \cG k_\sigma^2 \chi q^\beta \cG  \chi f(x) \right\rvert^2 \le \|f\|_{L^2(S)}^2 \int_{\Rm^d} \left| \int_{\Rm^d} G(x,y) k_\sigma^2 \chi(y) q^\beta(y) G(y,z) \chi (z) dy \right|^2 dz,
\end{equation*}
which implies
\begin{equation*}
 \|\cG k_\sigma^2 \chi q^\beta \cG \chi\|_{\mathcal{L}(L^2(S),L^2(K))}^2 \le \int_{\Rm^d \times K} \left| \int_{\Rm^d} G(x,y) k_\sigma^2 \chi(y) q^\beta(y) G(y,z) \chi(z)dy \right|^2 dz dx.
\end{equation*}
Taking expectations on both sides yields
\begin{equation*}
\begin{aligned}
 \E& \|\cG k_\sigma^2 \chi q^\beta \cG \chi\|_{\mathcal{L}(L^2(S),L^2(K))}^2\\
&\le  k_\sigma^4\int_{\Rm^{3d}\times K} G(x,y) G^*(x,y') G(y,z) G^*(y',z) R\left(\frac{y-y'}{\beta}\right) \chi(y) \chi(y') \chi^2(z)dy dy' dz dx\\
&\le  k_\sigma^4 \sup_{y \in S}\|G(y,\cdot)\|^2_{L^2(S)}\int_{\Rm^{2d}\times K} |R|\left(\frac{y-y'}{\beta}\right) \chi(y) \chi(y') |G(x,y)| |G(x,y')|dy dy' dx\\
&\le  k_\sigma^4 \sup_{y \in S}\|G(y,\cdot)\|^2_{L^2(S)} \sup_{x \in K}\|G(x,\cdot)\|^2_{L^2(S)}\beta^d \|R\|_{L^1}.
\end{aligned}
\end{equation*}
In the last line, we used Young's inequality. The second estimate of \fref{Glemma} and the symmetry of $G$ then completes the proof of \fref{eq:Gopnorm1}.

Regarding \eqref{eq:Gopnorm2}, we need fourth order moments estimate for mixing random fields. We have, for $q$ satisfying (A1)--(A4),
\begin{equation}
\label{eq:4moments}
\begin{aligned}
|\E[q^\beta_{\alpha_1}(y) q_{\alpha_2}^\beta(y')q_{\alpha_3}^\beta(z)q_{\alpha_4}^\beta(z')]| \le & \rho\left(\frac{y-z}{\beta}\right)\rho\left(\frac{y'-z'}{\beta}\right) + \rho\left(\frac{y-y'}{\beta}\right)\rho\left(\frac{z-z'}{\beta}\right) \\
& + \rho\left(\frac{y-z'}{\beta}\right)\rho\left(\frac{z-y'}{\beta}\right). 
\end{aligned}
\end{equation}
where $\rho = \varrho^{\frac 12}(\cdot/3)$ and $\alpha_j \in \{r,i\}$ for the real and imaginary parts of $q$; see \cite{Hairer} for a proof. Note that $\rho$ is bounded by one and integrable in $\Rm^d$ according to (A4). Decomposing $q$ into real and imaginary parts, there are 16 terms to estimate, all of which are treated in the same way with \fref{eq:4moments}. We therefore pick one as an example. Direct computation shows
\begin{equation*}
\begin{aligned}
\|\cG& k_\sigma^2 \chi q_r^\beta \cG k^2_\sigma \chi q_r^\beta  f\|_{L^2(K)}^2 \le k_\sigma^8\int_{K} d x  \int_{\Rm^{4d}} G(x,y) G^*(x,y')\\
&\times G(y,z)G^*(y',z')\chi(y)\chi(y')\chi(z)\chi(z')q_r^\beta(y)q_r^\beta(z)q_r^\beta(y')q_r^\beta(z') f(z) f^*(z')dz'dzdy'dy.
\end{aligned}
\end{equation*}
Taking expectation on both sides yield
\begin{equation*}
\begin{aligned}
\E \|\cG k_\sigma^2 \chi q_r^\beta \cG k^2 \chi q_r^\beta f\|_{L^2(K)}^2 \le  k_\sigma^8 & \int_{K\times \Rm^{4d}} G(x,y)G^*(x,y')G(y,z)G^*(y',z') \chi(y)\chi(y')\chi(z)\chi(z')\\
&\times \E[q_r^\beta(y)q_r^\beta(y')q_r^\beta(z)q_r^\beta(z')] f(z) f^*(z') dz'dzdy'dydx.
\end{aligned}
\end{equation*}
Following \fref{eq:4moments}, we need to control three integrals. The most singular one corresponds to the combination of arguments localizing the Green's functions on the diagonal, and is given by
\begin{equation*}
k_\sigma^8 \int_K  \left[\int_{\Rm^{2d}} \rho^\beta(y-z)|G|(y,z)|G|(x,y)\chi(y)\chi(z)|f(z)|dydz\right]^2 dx.
\end{equation*}
We then decompose $G$ as in Proposition \fref{propG} into $G(x,y) = \Phi(x-y) + p(x,y)$, where $\Phi$ is the free space Green's function solution to \fref{eq:FreeG} and $p$ a regular function. We split the integral above further into two terms. The first one is obtained by replacing the first $G$ by $\Phi$ above. Then, after a change of variables, this term becomes
\begin{align*}
  k_\sigma^8 \int_K  &\left[ \int_{\Rm^{d}} \rho^\beta(y)|\Phi|(y) \left(\int_{\Rm^d} |G|(x,y+z) \chi(y+z)\chi(z) |f(z)|dz \right) dy \right]^2 dx\\
&  \leq k_\sigma^8 |K| \| f\|^2_{L^2(S)} \sup_{x \in K}\| G(x,\cdot)\|^2_{L^2(S)} \left[ \int_{\Rm^{d}} \rho^\beta(y)|\Phi|(y) dy \right]^2. 
\end{align*}
When $d=3$, the last term is controlled by
\begin{equation*}
\int_{\Rm^3} \rho\left(\frac{y}{\beta}\right) \frac{e^{-\Im(k \underline{n}_e) |y|}}{4 \pi |y|} dy \le \beta^2 \int_{\Rm^3}  \frac{\rho(y)}{4 \pi |y|} dy \le C\beta^2.
\end{equation*}
Here, we used the fact that $\rho$ is integrable and bounded in $\Rm^d$. When $d=2$, we find, according to Lemma \ref{G2d} and (A4),
\begin{equation*}
C \int_{\Rm^2} \rho\left(\frac{y}{\beta}\right) (1+ |\log (\Re(k \underline{n}_e) |y|)|) dy \le C \beta^2 \int_{\Rm^2}  \rho(y) (1+ |\log (k \beta |y|)|)dy \le C\beta^2 (1+ |\log(k \beta)|).
\end{equation*}
The contribution of $p(y,z)$ is controlled as follows. We have, for all $x \in K$:
\begin{align*}
  \int_{\Rm^{2d}} \rho^\beta(y-z)& |p|(y,z) |G|(x,y) \chi(y)\chi(z) |f(z)|dydz \\
&\le \|p\|_{L^\infty_{y,z}(S \times S)} \int_{\Rm^d} (\rho^\beta * |\chi f|)(y) |G|(x,y)\chi(y)dy \\
  &\le \|p\|_{L^\infty_{y,z}(S \times S)}  \|\rho^\beta\|_{L^1}\|f\|_{L^2(S)} \sup_{x \in K}\| G(x,\cdot)\|_{L^2(S)}.
\end{align*}
We used Young's inequality in the last step. With \fref{plemma2} and \fref{Glemma}, the term above is controlled by $C \beta^d k^{3d/2-4} \kappa_e^{-2}$. Gathering all previous results, we find that the contribution associated with the first term in \eqref{eq:4moments} is of order $\sigma^4 k^{d+4} \kappa_e^{-2} \beta^4 (1+k^{2d-4} \kappa_e^{-2} \beta^2)$ for $d =3$ and $\sigma^4 k^{d+4} \kappa_e^{-2} \beta^4 ((1+|\log(\beta k)|)^2+k^{2d-4} \kappa_e^{-2})$ for $d=2$.

\medskip

The other terms in \eqref{eq:4moments} lead to weaker contributions in $\beta$. We take the following term as an example: after integrating in $x$,
\begin{align*}
  k_\sigma^8  \sup_{y \in S}\| G(\cdot,y)\|^2_{L^2(K)}\int_{\Rm^{4d}} &|G|(y,z)|G|(y',z') \chi(y)\chi(y')\chi(z)\chi(z') \\
                                                                      &\times  |f(z)||f(z')|\rho\left(\frac{y-y'}{\beta}\right)\rho\left(\frac{z-z'}{\beta}\right)d z'd zd y'dy 
\end{align*}
We recast the integral as
\begin{align*}
\int_{\Rm^{2d}} &\left(\rho^\beta*[|G|(\cdot,z)\chi ]\right)(y') \left(\rho^\beta*[|G|(y',\cdot)\chi |f|]\right)(z)  |f(z)| \chi(z)\chi(y') dy' dz\\
                &\leq \|f\|_{L^2(S)} \left( \int_{\Rm^d} \|\rho^\beta*[|G|(\cdot,z)\chi ]\|^2_{L^2} \; \|\rho^\beta*[|G|(y',\cdot)\chi(y')\chi |f| ](z)\|^2_{L^2_{y'}} \chi(z)dz   \right)^{1/2}\\
                &\leq \|f\|_{L^2(S)} \|\rho^\beta \|^2_{L^1} \sup_{z \in S}\| G(\cdot,z)\|_{L^2(S)}
                  \left( \int_{\Rm^d} \left| G(y',z')\chi(y') \chi(z')f(z') \right|^2 dy'dz'   \right)^{1/2}\\
  &\leq \|f\|^2_{L^2(S)} \|\rho^\beta \|^2_{L^1} \sup_{z \in S}\| G(\cdot,z)\|^2_{L^2(S)}.
\end{align*}
Multiplying by $k_\sigma^8  \sup_{y \in S}\| G(\cdot,y)\|^2_{L^2(K)}$, and accounting for \fref{Glemma}, this term generates of contribution of order $k^{2d} \sigma^4 \beta^{2d} \kappa_e^{-4}$, which is weaker than the other leading terms since $k \gg 1$.

By exactly the same method, we see that the third term on the right hand side of \eqref{eq:4moments} leads to a contribution of order $\beta^{2d}$. Combining all the results above, we proved \eqref{eq:Gopnorm2}.

\medskip

\medskip

\subsection{Asymptotics of the corrector} \label{secasymp}

We address here the convergence of the corrector $v^\beta=-\cG k_\sigma^2 \chi q^\beta u $ as $\beta \to 0$, keeping the other parameters fixed. We introduce first some mathematical preliminaries.

\paragraph{Preliminaries.} Let $\mathcal{H}$ be a Hilbert space, and $X : (\Omega,\mathscr{F},\bP) \to \mathcal{H}$ be a random variable with value in $\mathcal{H}$. Then $X$ induces a (Borel) probability measure on $\mathcal{H}$ as follows: for each Borel set $B \in \mathcal{H}$, define
\begin{equation*}
\lambda_X (B) := \bP(X \in B).
\end{equation*}
A family $\{X^\beta : \Omega \to \mathcal{H}\}_{\beta \in (0,1)}$ of random variables in $\mathcal{H}$ is said to converge in distribution to $X : \Omega \to \mathcal{H}$, as $\beta \to 0$, if the family of probability measures induced by $\{X^\beta\}$ on $\mathcal{H}$, denoted by $\lambda_{X^\beta}$, converges to that induced by $X$, denoted by $\lambda_X$. In other words, for any Borel measurable subset $B \subseteq \mathcal{H}$, the real numbers $\lambda_{X^\beta}(B)$ converges to $\lambda_X(B)$. An important and very useful criterion for the convergence in distribution of $X^\beta$ to $X$ is given as follows. We refer to \cite{partha} for its proof.

\begin{proposition}\label{prop:criterion} Suppose $\mathcal{H}$ is a separable Hilbert space. Let $\langle \cdot,\cdot \rangle$ denote the inner product on $\mathcal{H}$. Then $X^\beta$ converges in distribution to $X$, provided that
\begin{enumerate}
\item[(1)] for each $h \in \mathcal{H}$, the random variables $\langle X^\beta, h \rangle$ converges in distribution to $\langle X,h \rangle$. 
\item[(2)] the family of probability measures $\{\lambda_{X^\beta}\}$ is tight, i.e. $\forall \delta > 0$, there exists a compact set $\mathcal{K} \subseteq \mathcal{H}$, such that
\begin{equation*}
\inf_{\beta \in (0,1)} \lambda_{X^\beta}(\mathcal{K}) \ge 1 - \delta.
\end{equation*}
\end{enumerate}
\end{proposition}

For $K \subset \Rm^d$ a bounded set, we use the notation $H^s(K)$, $s \in (0,1)$, to denote the fractional Sobolev space defined by
\begin{equation*}
H^s(K) : = \left\{ f\in L^2(K) \,\Big|\, \|f\|_{L^2(K)} < \infty \text{ and } [f]_{H^s(K)} < \infty\right\},
\end{equation*}
where $[\cdot]_{H^s(K)}$ denotes the semi-norm given by
\begin{equation*}
\left( [f]_{H^s} \right)^2 := \int_{K \times K} \frac{|f(x) - f(y)|^2}{|x-y|^{d+2s}} dxdy.
\end{equation*}
The norm on $H^s(K)$ is given by the formula $\|f\|_{H^s}^2 := \|f\|^2_{L^2} + [f]_{H^s}^2$. 
In view of the fact that the embedding $H^s(K) \hookrightarrow L^2(K)$ is compact, a very simple yet useful criterion for tightness of measures on $L^2$ is given as follows; see \cite[Theorem A.2]{Jing16}.

\begin{proposition}\label{prop:tightness} Let $\{X^\beta\}_{\beta \in (0,1)}$ be a family of random variables on the separable Hilbert space $L^2(K)$. Suppose that there exists $C > 0$ and $s \in (0,1)$, independent of $\beta$, such that
\begin{equation*}
\E \|X^\beta\|_{H^s(K)} \le C.
\end{equation*}
Then the family $\{\lambda_{X^\beta}\}_{\beta \in (0,1)}$ of probability measures induced by $\{X^\beta\}$ is tight.
\end{proposition}

Next, we prove Theorem \ref{thm:clt}. We go back to the expansion formula \eqref{eq:exp1} and realize that, according to Theorem \ref{thm:homog},
\begin{equation*}
\lim_{\beta \to 0} \E \left\| \beta^{-\frac d2} \left(\cG k_\sigma^2 \chi q^\beta \cG k_\sigma^2 \chi q^\beta u + \cG k_\sigma^2 \chi q^\beta \cG k_\sigma^2 \chi q^\beta (u^\beta - u)\right) \right\|_{L^2(K)} = 0.
\end{equation*}
It follows that the limiting distribution in $L^2(K)$ of the normalized homogenization error $\beta^{-\frac d2}(u^\beta-u)$ is the same as that of $\beta^{-\frac d2}v^\beta$. In view of the criteria in Proposition \ref{prop:criterion} and Proposition \ref{prop:tightness}, it suffices to prove the results in Lemma \ref{lem:clt} below. Recall the notations in Theorem \ref{thm:clt}: $M^{1/2}=\{\alpha_{ij}\}$ is the square root of the matrix $M$ defined in \fref{defM}, and $W_y$ is defined by
\begin{equation*}
W_y:=(\alpha_{11} W_y^{(1)} + \alpha_{12} W_y^{(2)})+i(\alpha_{21}W_y^{(1)}  + \alpha_{22} W_y^{(2)}) = \begin{pmatrix} 1 & i\end{pmatrix} M^{\frac12}\mathbf{W}_y,
\end{equation*} 
where $W^{(1)}_y$ and $W^{(2)}_y$  are standard independent multiparameter ($y$-parameter) Wiener processes, and $\mathbf{W}_y:= (W^{(1)}_y, W^{(2)}_y)^T$. For any $\varphi \in L^2(K)$, let finally $m(y)=\int_{K} G(x,y) \varphi(x) dy$, where $G$ is the Green's function of the homogenized equation \ref{eq:hpde-s1}. We have then:
\begin{lemma}\label{lem:clt}
Under the same assumptions of Theorem \ref{thm:clt}, we have:
\begin{enumerate}
\item[(1)] For each $\varphi \in L^2(K)$,
\begin{equation}
\label{eq:kom1}
\left(\frac{v^\beta}{\beta^{d/2}}, \varphi \right) \xrightarrow[\beta\to 0]{\mathrm{distribution}} -\sigma k^2  \int_{\Rm^d} u(y)m(y)\chi(y) dW_y.
\end{equation}
\item[(2)] Let $d = 2,3$ and $s < (4-d)/2$. Then there exists constant $C > 0$ such that
\begin{equation}
\label{eq:kom2}
\E \|\beta^{-d/2} v^\beta\|_{H^s(K)} \le C.
\end{equation}
\end{enumerate}
\end{lemma}

\begin{proof} For item one, we compute
\begin{equation*}
\begin{aligned}
\beta^{-d/2}(v^\beta, \varphi) &= \frac{-\sigma k^2 }{\sqrt{\beta^d}} \int_{K \times \Rm^d} G(x,y)\chi(y) q\left(\frac{y}{\beta}\right) u(y) \varphi(x) dy dx\\
&= \frac{-\sigma k^2}{\sqrt{\beta^d}} \int_{\Rm^d}  q\left(\frac{y}{\beta}\right) u(y) m(y) \chi(y) dy.
\end{aligned}
\end{equation*}
Write then for simplicity $m_0(y)=u(y) m(y) \chi(y)$. We address the convergence of $\zeta^\beta:=\beta^{-d/2}(v^\beta, \varphi)$ by considering the real vector $(\Re(\zeta^\beta),\Im(\zeta^\beta))^T$. For this, we consider random vector $h^\beta$ of the form, with $Q^\beta(y)=(q_r(\frac{y}{\beta}),q_i(\frac{y}{\beta}))^T$, 
$$
h^\beta= \frac{-1}{\sqrt{\beta^d}} \int_{\Rm^d} A(y) Q^\beta(y)dy,
$$
where $A=\{a_{ij}\}$ is matrix whose entries are real $L^2(S)$ functions supported in $S$. These are integrals involving the product of an $L^2(S)$ function and the random oscillatory function $q_\alpha^\beta$, $\alpha=\{r,i\}$, which is a rescaled version of the short range correlated random field $q_\alpha(x,\omega)$. Adapting the results of \cite{GB-08,Bal-Jing} to our case, we find that
\begin{equation*}
h^{\beta} \xrightarrow[\beta \to 0]{\mathrm{distribution}} h \sim \calN \left(0, \int_S A(y) M A^T(y) dy \right),
\end{equation*}
where $M$ is defined in \fref{defM}. With the multiparameter Wiener process $\mathbf{W}_y = (W^{(1)}_y,W^{(2)}_y)^T$ defined earlier, we find that the distribution on the right hand side is realized by
$$
\int_S A(y) M^{1/2} d\bold{W}_y \sim \calN \left(0, \int_S A(y) M A^T(y) dy \right).
$$
The left hand side can be written as a sum of several Wiener integrals in terms of the independent components $W^{(1)}_y$ and $W^{(2)}_y$; those integrals are simply random normal variables with zero mean and they have appropriate joint variances.

To verify \eqref{eq:kom1}, it suffices to write $\zeta^\beta$ as
$$
\zeta^\beta=\frac{-\sigma k^2}{\sqrt{\beta^d}} \int_{\Rm^d} A(y) Q^\beta(y)dy, \qquad A=\left(
\begin{array}{ll}
\ds \Re(m_0)& - \Im (m_0) \\
\ds \Im(m_0)& \Re(m_0)
\end{array}
\right),
$$
so that
\begin{equation*}
\zeta^{\beta} \xrightarrow[\beta \to 0]{\mathrm{distribution}}
-\sigma k^2 \int_S A(y) M^{1/2} d\bold{W}_y.
\end{equation*}
 We then obtain \fref{eq:kom1} by inspection.

Next, we show that the measures generated by $v^\beta/\beta^{d/2}$ are tight in $L^2(K)$, by proving \eqref{eq:kom2}. Here, $K$ is any bounded set in the upper half space. By definition, $\|v^\beta\|_{H^s(K)}^2 = \|v^\beta\|_{L^2(K)}^2 + [v^\beta]_{H^s(K)}^2$ where the second term is the $H^s$ seminorm given by
\begin{equation*}
\begin{aligned}
{[v^\beta]}^2_{H^s(K)} &= \left[-\cG k_\sigma^2 \chi q^\beta u\right]_{H^s}^2 = \int_{K \times K} \frac{\left| \cG k_\sigma^2 \chi q^\beta u (x) - \cG k_\sigma^2 \chi q^\beta u(y) \right|^2}{|x-y|^{d+2s}} \, dx dy\\
 &= k_\sigma^4 \int_{K \times K} \frac{1}{|x-y|^{d+2s}} \left|\int_{\Rm^d} \left[G(x,z) - G(y,z)\right] \chi(z) q\left(\frac{z}{\beta}\right) u(z) dz \right|^2 \, dx dy.
 \end{aligned}
\end{equation*}
Taking expectations, we have
\begin{equation}
\label{eq:Hs1}
\begin{aligned}
\E [v^\beta]_{H^s}^2 = &\int_{K\times K \times \Rm^{2d}}  \frac{\left[G(x,\xi) - G(y,\xi)\right] \left[G^*(x,\eta) - G^*(y,\eta)\right]}{|x-y|^{d+2s}} R\left(\frac{\xi - \eta}{\beta}\right)\\
&\qquad\qquad \times \chi(\xi)\chi(\eta) u(\xi)  u^*(\eta)\, d\xi d\eta d x dy.
\end{aligned}
\end{equation}
Then, with the Cauchy-Schwarz inequality, 
$$
\begin{aligned}
& \sup_{\xi, \eta \in S} \int_{K \times K}\frac{\left|G(x,\xi) - G(y,\xi)\right|\, \left| G(x,\eta) - G(y,\eta)\right|}{|x-y|^{d+2s}} dx dy\\
& \qquad \le  \sup_{\xi \in S} \int_{K \times K}\frac{\left|G(x,\xi) - G(y,\xi)\right|^2}{|x-y|^{d+2s}} dx dy \leq \sup_{\xi \in S}\|G(\cdot,\xi)\|^2_{H^s(K)}.
\end{aligned}
$$
With \fref{estHsG} and plugging the latter estimate in \eqref{eq:Hs1}, we have
\begin{equation*}
\E [v^\beta]_{H^s}^2 \le \|G(x,\xi)\|_{L^\infty_\xi H^s_x}^2 \int_{\Rm^{2d}} |R|\left(\frac{\xi - \eta}{\beta}\right)|\chi(\xi) u(\xi) \chi(\eta) u(\eta)|\, d \xi d\eta.
\end{equation*}
By similar arguments as in the proof of Lemma \ref{lem:cGL2}, we get to conclude that
\begin{equation*}
\E [v^\beta]_{H^s}^2 \le C\beta^d \|u\|^2_{L^2(S)}.
\end{equation*}
Together with Lemma \ref{lem:cGL2} which controls the $L^2$ norm of $v^\beta$, this completes the proof.
\end{proof}

\medskip

An immediate result of the lemma is that
\begin{equation*}
\frac{v^\beta}{\sqrt{\beta^d}} \xrightarrow[\beta \to 0]{\mathrm{distribution}} -\sigma k^2  \int_{\Rm^d} G(x,y) \chi(y) u(y) dW_y \qquad \text{in } L^2(K).
\end{equation*}
Since $\beta^{-\frac d2}(u^\beta- u)$ has the same limiting distribution, the proof of Theorem \ref{thm:clt} is complete.

\subsection{Proof of Proposition \ref{propG}} \label{secpropG}

Write $x=(x',z)$ with $x'=(x_1,x_{2})$ when $d=3$, and $x'=x_1$ when $d=2$. Taking the Fourier transform of \fref{eq:homG} in the $x'$ variable (with dual variable $\xi \in \Rm^{d-1}$), we find
  \be \label{fourG}
  \partial_z^2 \hat G+k^2(z,\xi) \hat G=e^{i x'_0 \cdot \xi} \delta(z-z_0), \qquad z \in \Rm,\ee
  where $x_0=(x_0',z_0)$ and the complex number $k^2(z,\xi)$ is defined by
  $$
  k^2(z,\xi)=\left\{
    \begin{array}{ll}
      k_0^2(\xi)=k^2 \underline n_0^2(\xi)-|\xi|^2=k^2(n_0^2+ i \alpha)-|\xi|^2 & z >0\\
      k_1^2(\xi)=k^2 \underline n_e^2(\xi)-|\xi|^2=k^2(n_e^2+ i (\kappa_e+\alpha))-|\xi|^2& z \in (-L,0)\\
      k_2^2(\xi)=k^2 \underline n_2^2(\xi)-|\xi|^2=k^2(n_2^2+ i (\kappa_2+\alpha))-|\xi|^2& z<-L.
    \end{array}
    \right.
    $$
We write the dependencies of $\hat G$ as $\hat G(z,z_0,\xi,x_0')$. Although it is possible to obtain exact expressions for $\hat G$, we follow a different route that seems somewhat more direct and would apply if the refractive index in the sea ice layer depends also on $z$. We start with the case $z_0 \in (-L,0)$. 
  \paragraph{Estimates in the layer.}
When $z_0 \in(-L,0)$ and $z>0$, $\hat G$ has the form $\hat G(z,z_0,\xi,x_0')= A e^{i k_0 z}$, for some constant $A$. In the latter expression, $k_0$ is the principal square root of $k_0^2$, and we selected the outgoing solution of \fref{fourG}. In the same way, when $z_0 \in(-L,0)$ and $z<-L$, we have $\hat G(z,z_0,\xi,x_0')= A' e^{-i k_2 (z+L)}$. Since $z_0 \in (-L,0)$, it follows from \fref{fourG} that $\hat G$ and $\partial_z \hat G$ are continuous functions at $z=0$ and $z=-L$, and we therefore obtain the boundary conditions
    $$
    \partial_z \hat G(0,\xi)=i k_0 \hat G(0,\xi), \qquad \partial_z \hat G(-L,\xi)=-i k_2 \hat G(-L,\xi).
    $$
Let $\Phi_\xi$ be the Green's function of the one-dimensional Helmholtz equation with wavenumber $k_1(\xi)$, i.e. $\Phi_\xi(z)=e^{i k_1(\xi) |z|}/(2 i k_1(\xi))$. It is the Fourier transform of $\Phi(x)$ defined in \fref{eq:FreeG} and  verifies
    $$
    \partial_z^2 \Phi_\xi(z)+k_1^2(\xi) \Phi_\xi(z)=\delta(z), \qquad z \in \Rm.
    $$
With $z_0 \in (-L,0)$, define $\hat p$ by $e^{i x'_0 \cdot \xi} \hat p(z,z_0,\xi,x_0')=\hat G(z,z_0,\xi,x_0')-e^{i x'_0 \cdot \xi} \Phi_\xi(z-z_0)$. We have then 
\be \label{eqp}
  \partial_z^2 \hat p+k^2(z,\xi) \hat p=0, \qquad z \in (-L,0).
\ee
To equip the above equation with boundary conditions, we use again the continuity of $\hat{G}$ and $\partial_z \hat{G}$ at $z = 0$ and $z = -L$. Moreover, since $\partial_z |z|= 2H(z)-1$ for $z \neq 0$ and $H$ is the Heaviside function, we find the following boundary conditions for $\hat p$:
$$
\left\{
\begin{array}{l}
  \partial_z \hat p(0,z_0,\xi,x_0')=i k_0 \hat p(0,z_0,\xi,x_0')+ i (k_0 - k_1(2H(-z_0)-1)) \Phi_\xi(z_0)\\
\partial_z \hat p(-L,z_0,\xi,x_0')=-i k_2 \hat p(-L,z_0,\xi,x_0')-i (k_2 + k_1(2H(-z_0-L)-1)) \Phi_\xi(z_0+L).
\end{array}
\right.
$$
We have $H(-z_0)=1$ and $H(-z_0-L)=0$ for all $z_0 \in (-L,0)$, and the latter boundary conditions simplify when $z_0 \in (-L,0)$. Multiplying then \fref{eqp} by $\hat p^*$, integrating by parts in $(-L,0)$ using the above boundary conditions, and taking the real and imaginary parts lead to, for all $z_0 \in (-L,0)$ (we only make explicit the dependency of $\hat p$ on $z$ for simplicity),
\begin{align} \nonumber
\Re (k_0) |\hat p(0)|^2 &+\Re (k_2) |\hat p(-L)|^2+k^2 (\kappa_e+\alpha) \int_{-L}^0 |\hat p(z)|^2 dz \\
&= \label{rel1} 
\Re\left((k_1-k_0) \Phi_\xi(z_0) \overline{\hat p(0)} \right)+\Re\left((k_1-k_2) \Phi_\xi(z_0+L) \overline{\hat p(-L)} \right)\\
 \nonumber \Im (k_0) |\hat p(0)|^2 &+\Im (k_2) |\hat p(-L)|^2+(|\xi|^2-k^2 n_e^2) \int_{-L}^0 |\hat p(z)|^2 dz + \| \partial_z \hat p\|_{L^2(-L,0)}^2\\
&=-\label{rel2}
\Im\left((k_1-k_0) \Phi_\xi(z_0) \overline{\hat p(0)} \right)-\Im\left((k_1-k_2) \Phi_\xi(z_0+L) \overline{\hat p(-L)} \right).
\end{align}
Since $\Im(k_1) \geq 0$, we have $|\Phi_\xi(z)| \leq 1/(2 |k_1|)$ and we find, using the Cauchy-Schwarz inequality in the r.h.s of \fref{rel1}-\fref{rel2},
\begin{align}\label{eg1}
&\Re (k_0) |\hat p(0)| ^2 +\Re (k_2) |\hat p(-L)|^2 +2 k^2 (\kappa_e+\alpha) \int_{-L}^0 |\hat p(z)|^2 dz\leq \frac{|k_1-k_0|^2}{\Re (k_0) |k_1|^2}+\frac{|k_1-k_2|^2}{\Re (k_2) |k_1|^2}\\
&\label{eg2}\Im (k_0) |\hat p(0)| ^2 +\Im (k_2) |\hat p(-L)|^2 +2 (|\xi|^2-k^2 n_e^2) \int_{-L}^0 |\hat p(z)|^2 dz\leq \frac{|k_1-k_0|^2}{\Im (k_0) |k_1|^2}+\frac{|k_1-k_2|^2}{\Im (k_2) |k_1|^2}.
\end{align}
We also have the direct inequality, obtained from \fref{rel1},
\be \label{eg3}
 k^2 (\kappa_e+\alpha) \int_{-L}^0 |\hat p(z)|^2 dz \leq \frac{|k_1-k_0|}{2|k_1|}|\hat p(0)|+\frac{|k_1-k_2|}{2|k_1|}|\hat p(-L)|, \qquad \forall \xi \in \Rm^{d-1}.
 \ee
We now exploit the previous inequalities to estimate $p$ in the layer and start with the $L^2(\Rm^{d-1} \times (-L,0))$ norm.\\

 \textit{$L^2$ estimates.} Inequality \fref{eg1} allows us to estimate $\hat p$ for $|\xi| \leq k n_0$ in $L^2$ as follows: with the change of variables $\xi \to k \xi$ in the r.h.s, we obtain
\be \label{expp} 
2 k^2 \kappa_e \int_{|\xi| \leq k n_0}\int_{-L}^0 |\hat p(z,\xi)|^2 dz d\xi
\leq k^{d-2}\int_{|\xi| \leq n_0} \left(\frac{|\underline n_e-\underline n_0|^2}{\Re(\underline n_0)|\underline n_e|^2}+\frac{|\underline n_e-\underline n_2|^2}{\Re(\underline n_2)|\underline n_e|^2}\right) (\xi) d\xi,
\ee
with
$$
\underline{n}_0^2(\xi)=n_0^2-|\xi|^2 + i \alpha, \qquad \underline{n}_e^2(\xi)=n_e^2-|\xi|^2+ i (\kappa_e+\alpha), \qquad \underline n_2^2(\xi)=n_2^2-|\xi|^2+ i (\kappa_2+\alpha).
    $$
We only treat the term in $\underline{n}_0$ in the r.h.s of \fref{expp} since the other one is handled similarly. Some direct algebra involving Taylor expansions shows that there exists $C>0$ that is independent of $\xi$ and $\alpha$ such that 
\be \label{estdn}
|(\underline n_e-\underline n_0)(\xi)| +|(\underline n_e-\underline n_2)(\xi)| \leq C, \qquad \forall \xi \in \Rm^{d-1}.
\ee
With $\kappa_\alpha=\kappa_e+\alpha$, we are then left with estimating the integral 
\begin{align*}
\int_{|\xi| \leq n_0} &\frac{d \xi}{\Re(\underline n_0(\xi))  |\underline n_e(\xi)|^2}=  \int_{|\xi| \leq n_0 } \frac{d \xi}{\sqrt{\sqrt{(n_0^2-|\xi|^2)^2+\alpha^2}+(n_0^2-|\xi|^2)}\sqrt{(n_e^2-|\xi|^2)^2+\kappa_\alpha^2)}} \\
                      & \leq \frac{1}{\kappa_\alpha} \int_{|\xi| \leq n_0 } \frac{d \xi}{\sqrt{n_0^2-|\xi|^2}} \leq \frac{C}{\kappa_\alpha},
\end{align*}
where $C$ is independent of $\kappa_e$ and $n_0$. Above, we used the fact that $n_0<n_e$. This gives us a bound for $\hat p$ in $L^2$ in the region where $|\xi| \leq k n_0$. Assume now without lack of generality that $n_e>n_2$ (the opposite case $n_e<n_2$ can be handled similarly). For $|\xi| \geq k n_e$, we combine \fref{eg2} and  \fref{eg3} to arrive at
$$
 C k^2 (\kappa_e+\alpha) \int_{-L}^0 |\hat p(z)|^2 dz \leq \frac{|k_1-k_0|^2}{\Im (k_0) |k_1|^2}+\frac{|k_1-k_2|^2}{\Im (k_2) |k_1|^2}.
$$
The first term in the r.h.s is easily taken care of by \fref{estdn} and the same argument as before. For the second one, with the notation $\kappa_{2,\alpha}=\kappa_2+\alpha$, we need to estimate the integral: 
\begin{equation*}
\int_{|\xi| \geq n_e} \frac{d \xi}{\Im(\underline n_2(\xi))  |\underline n_e(\xi)|^2}=  \int_{|\xi| \geq n_e } \frac{d \xi}{\left(\sqrt{(n_2^2-|\xi|^2)^2+\kappa_{2,\alpha}^2}-(n_2^2-|\xi|^2)\right)^{1/2}\sqrt{(n_e^2-|\xi|^2)^2+\kappa_\alpha^2)}}.
\end{equation*}
Using the change of variable $|\xi| \to \kappa_\alpha r + n_e$, we see that the above is bounded by
\begin{equation*}
\leq C \int_{0}^\infty \frac{ (\kappa_\alpha r+ n_e)^{d-2}d r}{\sqrt{(\kappa_\alpha r+n_e+n_2)(\kappa_\alpha r+n_e-n_2)}\sqrt{(((\kappa_\alpha r+2 n_e)r)^2+1)}} \leq C (1+\kappa_e^{- \delta}),
\end{equation*}
for any $\delta>0$. It remains to treat the case $k n_0 \leq |\xi| \leq k n_e$. We start by using \fref{eg3} to obtain, after multiplying by $k^2 n_e^2-|\xi|^2$, 
$$
\int_{-L}^0 |\hat p(z)|^2 dz\leq \frac{(k^2 n_e^2-|\xi|^2)}{ k^2 (\kappa_e+\alpha)}\left(\frac{|k_1-k_0|}{2|k_1|}|\hat p(0)|+\frac{|k_1-k_2|}{2|k_1|}|\hat p(-L)|\right).
$$
Using \fref{eg2} and Young's inequality, we get
\be \label{eqim}
\Im (k_0) |\hat p(0)| ^2 +\Im (k_2) |\hat p(-L)|^2 \leq C\left(\frac{(k^2 n_e^2-|\xi|^2)}{ k^2 (\kappa_e+\alpha) }\right)^2 \left(\frac{|k_1-k_0|^2}{\Im (k_0) |k_1|^2}+\frac{|k_1-k_2|^2}{\Im (k_2) |k_1|^2}\right),
\ee
which, together with \fref{eg3}, leads to
$$
  k^2 (\kappa_e+\alpha) \int_{-L}^0 |\hat p(z)|^2 dz \leq C \left(\frac{(k^2 n_e^2-|\xi|^2)}{ k^2 (\kappa_e+\alpha) }\right) \left(\frac{|k_1-k_0|^2}{\Im (k_0) |k_1|^2}+\frac{|k_1-k_2|^2}{\Im (k_2) |k_1|^2}\right).
$$
  Integrating over $k n_0 \leq |\xi| \leq k n_e$, and using that $(k^2 n_e^2-|\xi|^2)/|k_1|^2 \leq 1$ together with \fref{estdn}, we find 
$$
 k^{4-d} (\kappa_e+\alpha)^2 \int_{k n_0 \leq |\xi| \leq k n_e} \int_{-L}^0 |\hat p(z)|^2 dz d\xi \leq C \int_{ n_0 \leq |\xi| \leq n_e}\left(\frac{1}{\Im (\underline n_0)}+\frac{1}{\Im (\underline n_2)}\right) d \xi \leq C.
$$
 Collecting results, we have proven \fref{plemma}. We turn now to the $L^\infty$ estimates.\\

\textit{$L^\infty$ estimates.} We control $\|p\|_{L^\infty}$ by $\|\hat{p}\|_{L^1}$. For $z \in [-L,0]$, the solution to \fref{eqp} reads
$$
\hat p(z)=\frac{\sin(k_1(z+L))}{\sin(k_1L)}\hat p(0)-\frac{\sin(k_1z)}{\sin(k_1L)}\hat p(-L)
$$
With $\Im(k_1)>0$, the function $| \sin(k_1 z)|$ is strictly increasing with $|z|$, so that
$$
\left|\frac{\sin(k_1(z+L))}{\sin(k_1L)}\right| \leq 1, \qquad \left|\frac{\sin(k_1z)}{\sin(k_1L)}\right| \leq 1,
$$
and as a consequence $
|\hat p(z)| \leq |\hat p(0)|+|\hat p(-L)|
$. As before, assuming without lack of generality that $n_e>n_2$, we split the integration into three parts, $|\xi| \leq k n_0$, $k n_0 \leq |\xi| \leq kn_e$, and $|\xi|>k n_e$. We estimate the part $|\xi| \leq k n_0$ using \fref{eg1}. We only treat $\hat p(0)$ and the term involving $k_0$ in the r.h.s since all other calculations are similar. Using \fref{estdn} and the fact that $n_0<n_e$, we then find
$$
k^{d-2}\int_{|\xi| \leq n_0}\frac{|\underline n_e-\underline n_0|}{\Re (\underline n_0) |\underline n_e|} d\xi \leq Ck^{d-2}\int_{|\xi| \leq n_0}\frac{d\xi}{\Re (\underline n_0)} d\xi \leq C k^{d-2}.
$$
For the part $|\xi| \geq k n_e$, we only treat as above the $k_0$ part of the r.h.s in $\hat p(0)$ and write
\be \label{xip}
k^{d-2}\int_{|\xi| \geq n_e}\frac{|\underline n_e-\underline n_0|}{\Im (\underline n_0) |\underline n_e|} d\xi =k^{d-2} \int_{n_e \leq |\xi| \leq R}(\cdots)d\xi+k^{d-2} \int_{|\xi| \geq R}(\cdots)d\xi,
\ee
for some $R>0$ that will be defined below. The first term is easily estimated using \fref{estdn}:
$$
C k^{d-2} \int_{n_e \leq |\xi| \leq R} \frac{d\xi}{|n_e|} \leq C k^{d-2}.
$$
For the second term, we need to refine estimate \fref{estdn} for large $|\xi|$, and find, for appropriate $C_{R}$ and $R$ sufficiently large, 
\be \label{estdn2}
|(\underline n_e-\underline n_0)(\xi)| +|(\underline n_e-\underline n_2)(\xi)| \leq C_{R} |\xi|^{-1}, \qquad \forall |\xi| \geq R.
\ee
With the change of variables $|\xi| \to \kappa_\alpha r+n_e $ in the second term of \fref{xip}, we find, using \fref{estdn2},
\begin{align*}
C \kappa_\alpha^{1/2} \int_{0}^\infty \frac{ (\kappa_\alpha r+ n_e)^{d-2}d r}{\sqrt{(\kappa_\alpha r+n_e+n_2)(\kappa_\alpha r+n_e-n_2)}(((\kappa_\alpha r+2 n_e)r)^2+1)^{1/4} (\kappa_\alpha r+n_e)} \leq C.
\end{align*}
The last piece $k n_0 \leq |\xi| \leq k n_e$ is estimated by using \fref{eqim}. Since $(n_e^2-|\xi|^2)/|n_e| \leq (n_e^2-|\xi|^2)^{1/2}$, the $k_0$ part in the r.h.s in $\hat p(0)$ is bounded by 
$$
k^{d-2}\int_{n_0 \leq |\xi| \leq n_e}\frac{n_e^2-|\xi|^2}{\kappa_e+\alpha}\frac{|\underline n_e-\underline n_0|}{\Im (\underline n_0) |\underline n_e|} d\xi
\leq k^{d-2}\int_{n_0 \leq |\xi| \leq n_e}\frac{(n_e^2-|\xi|^2)^{1/2}}{\kappa_e}\frac{d\xi }{\Im (\underline n_0)} d\xi \leq C k^{d-2} \kappa_e^{-1}.
$$
This proves the first estimate in \fref{plemma2}. The second one follows from the observation that, for $z>0$ and $z_0 \in (-L,0)$,
$$
|\hat p(z)|=| \hat p(0) e^{i k_0(\xi) z}| \leq |\hat p(0)|.
$$

We turn now to the case $z_0>0$ and prove the second estimate in \fref{Glemma}. The calculations are very similar and we only provide the main ideas.

\paragraph{Estimates for $z_0$ in the upper half-space.} We start with \fref{fourG} with now $z_0>0$, and derive first boundary conditions at $x=0$ and $x=-L$. The one at $x=-L$ is the same as when $z_0 \in (-L,0)$. For the one at $x=0$, we notice that $\partial_z \hat G$ has a jump of value $e^{i x_0' \cdot \xi}$ at $z=z_0$. Accounting for this, solving \fref{fourG} for $z>z_0$ and $z \in (0,z_0)$, and eliminating the unknown reflection coefficient at $x=0$  yields the boundary condition
$$
\partial_z \hat G(0)= i k_0 \hat G(0)-e^{i (k_0 z_0+x_0'\cdot \xi)}.
$$
In the region $z \in (-L,0)$, $\hat{G}$ satisfies the Helmholtz equation with  wavenumber $k_1$ and has the boundary conditions specified above. We can then estimate $\hat{G}$ in $L^2$ by reproducing the steps for the analysis of $\hat p$ in the previous section. In particular, we replace, in \fref{rel1}-\fref{rel2}, $i(k_1-k_0) \Phi_\xi(z-z_0)$ by $e^{i (k_0 z_0+x_0'\cdot \xi)}$, $\hat{p}$ by $\hat{G}$ and set $k_1=k_2$. We then obtain for the parts $|\xi| \leq k n_0$ and $|\xi| \geq k n_e$, 
$$
2 k^2 \kappa_e \int_{|\xi| \leq k n_0}\int_{-L}^0 |\hat G(z,\xi)|^2 dz d\xi
\leq C k^{d-2}\int_{|\xi| \leq n_0} \frac{e^{-2 k \Im(\underline n_0) z_0}}{\Re(\underline n_0)}d\xi \leq C k^{d-2},
$$
and
\begin{align*}
2 k^2 \kappa_e \int_{|\xi| \geq k n_e}\int_{-L}^0 |\hat G(z,\xi)|^2 &dz d\xi
                                                                      \leq C k^{d-2}\int_{|\xi| \geq n_e} \frac{e^{-2 k \Im(\underline n_0) z_0}}{\Im(\underline n_0)}d\xi \\
  &\leq
C k^{d-2}\int_{|\xi| \geq n_e} \frac{e^{-2 k \sqrt{|\xi|^2-n_0^2} z_0}}{\sqrt{|\xi|^2-n_0^2}}d\xi  \leq C k^{d-2} (1+k^{2-d}).
\end{align*}
The last piece $k n_0 \leq |\xi| \leq k n_e$ is estimated by 

$$
 k^{4-d} (\kappa_e+\alpha)^2 \int_{k n_0 \leq |\xi| \leq k n_e} \int_{-L}^0 |\hat G(z,\xi)|^2 dz d\xi \leq C \int_{ n_0 \leq |\xi| \leq n_e}\frac{e^{-2 k \Im(\underline n_0) z_0}}{\Im (\underline n_0)}\leq C.
 $$
 Gathering previous estimates, we obtain the second part of \fref{Glemma}.

We continue the proof of lemma by estimating the free Green's function $\Phi_\xi$.
  \paragraph{Estimates on $\Phi_\xi$.}

In order to obtain the first estimate in \fref{Glemma}, we compute
\bee
\|\Phi_\xi(\cdot - z_0) \|^2_{L^2(S)}&=&\int_{\Rm^{d-1}}\int_{-L}^0
\frac {e^{- 2\Im{(k_1(\xi))}|z-z_0|}}{4 |k_1(\xi)|^2} d\xi dz \leq C \int_{\Rm^{d-1}}
\frac {d \xi}{\Im(k_1(\xi)) |k_1(\xi)|^2}\\
&\leq & C k^{d-4} \int_{\Rm^{d-1}}
\frac {d \xi}{\Im(\underline n_e(\xi)) |\underline n_e(\xi)|^2}.
\eee
With the change of variables $|\xi| \to \kappa_\alpha r +n_e$, we find
\begin{align*}
\int_{\Rm^{d-1}} &\frac{d \xi}{\Im(\underline n_e(\xi))  |\underline n_e(\xi)|^2}=  \int_{\Rm^{d-1}} \frac{d \xi}{\sqrt{\sqrt{(n_e^2-|\xi|^2)^2+\kappa_\alpha^2}-(n_e^2-|\xi|^2)}\sqrt{(n_e^2-|\xi|^2)^2+\kappa_\alpha^2)}} \\
 &=C \kappa_\alpha ^{-1/2} \int_{-n_e/\kappa_\alpha}^\infty \frac{ (\kappa_\alpha r+ n_e)^{d-2}d r}{\sqrt{\sqrt{((\kappa_\alpha r+2 n_e)r)^2+1}+(\kappa_\alpha r+2 n_e)r}\sqrt{(((\kappa_\alpha r+2 n_e)r)^2+1)}}. \end{align*}
It is direct to see that the above integral is finite for $r>0$ and then bounded by $C \kappa_e^{-1/2}$. We focus therefore on the piece $r \in (-n_e / \kappa_e,0)$. Since the function $x \to \sqrt{x^2+1}-x$ is decreasing, the integral is readily bounded by $C_R \kappa_e^{-1/2}$ for $r \in (-R, 0)$. For $R$ sufficiently large, there is a constant $C_R$ such that
\be
\sqrt{x^2+1}-x \geq C_R /x, \qquad \forall x \geq R. \label{LB}
\ee
If $n_e / \kappa_\alpha \leq R$ we are done and the integral is controlled by $C \kappa_e^{-1/2}$. If not, we estimate the integral by, using \fref{LB} and the decrease of $x \to \sqrt{x^2+1^2}-x$, since $(2n_e-\kappa_\alpha r)r \leq (2n_e-\kappa_\alpha R) r:=c_R r$ for $r \in (R, n_e / \kappa_\alpha)$,
\bee
C \kappa_e ^{-1/2}  \int_{R}^{n_e/\kappa_\alpha} \frac{d r}{\sqrt{\sqrt{(c_R r)^2+1}-c_R r}\sqrt{((n_e r)^2+1)}} \leq C \kappa_e ^{-1/2} \int_{R}^{n_e/\kappa_\alpha} dr/r^{1/2} \leq C \kappa_e^{-1}.
\eee
Collecting estimates, and using $\kappa_e \leq 1$, we finally arrive at
$$
\sup_{z_0 \in \Rm} \|\Phi_\xi(\cdot - z_0) \|_{L^2(S)} \leq C k^{d/2-2} \kappa_e^{-1/2}.
$$

Together with \fref{plemma}, this yields the first estimate in \fref{Glemma}.

\paragraph{Estimate in $H^s(K)$.} We write $G(x,y)=\Phi(x-y)+p(x,y)$ and start with $\Phi$, for which a direct computation shows that it does not belong to the Sobolev space $H^1(\Rm^d)$. Nevertheless, we can check that for any $s < (4-d)/2$, we have $\Phi \in H^s(\Rm^d)$. Indeed, since we have from \fref{eq:FreeG} that
\begin{equation*}
\mathcal{F}_{x\mapsto \xi}[\Phi](\xi) = \frac{1}{-|\xi|^2 + k^2\underline n_e^2},
\end{equation*}
and we find
\begin{equation*}
  \int_{\Rm^d} (1+|\xi|^2)^s \left\lvert \mathcal{F}_{x\mapsto \xi}[\Phi](\xi)\right\rvert^2 d\xi = \int_{\Rm^d} (1+|\xi|^2)^s \frac{1}{|-i k^2(\kappa_e+\alpha)-k^2n_e^2 + |\xi|^2|^2} d\xi.
\end{equation*}
Since $\kappa_e+\alpha > 0$, we see that the integral above is finite as long as $4-2s > d$, or equivalently, $s < (4-d)/2$. This shows that $\Phi \in H^s(\Rm^d) \subset H^s(K)$. Regarding $p$, it satisfies the equation
$$
\Delta_x p+ k^2 \underline n^2 p= k^2(\underline n_e^2-\underline n^2) \Phi(x-y). 
$$
\medskip
With \fref{plemma2} and the fact that $\Phi \in L^2(\Rm^d)$, it follows that $\sup_{y \in S}\| \Delta_x p(\cdot,y)\|_{L^2(K)}$ is finite, which proves \fref{estHsG} by elliptic regularity.

\subsection{Other proofs}
\subsubsection{Proof of Lemma \ref{G2d}} \label{proofG2d}

  The zero order Hankel function $H_0^{(1)}(z)$ admits the expression, for $z \in \Cm$ with $\Re(z) > 0$ and $\Im(z)>0$,
   $$
  H_0^{(1)}(z)=\frac{1}{\pi} \int_0^\pi e^{i z \sin t} dt -\frac{2 i}{\pi} \int_0^\infty e^{-z \sinh t} dt:=H_1(z)+H_2(z).
  $$
  Since $\Im(z) \geq 0$, we have the direct estimate $|H_1(z)| \leq 1$. Furthermore,
  $$
   H_2(z)=-\frac{2 i}{\pi} \int_0^1 e^{-z \sinh t} dt-\frac{2 i}{\pi} \int_1^\infty e^{-z \sinh t} dt:=H_3(z)+H_4(z),
   $$
   with $|H_3(z)| \leq 2 /\pi$. 
   Finally, since $2 \sinh(t) \geq e^t-e$, for all $t \geq 1$, we have
\bee
   |H_4(z)| &\leq& \frac{2 e^{\frac{1}{2} e \Re(z)}}{\pi} \int_1^\infty e^{- \frac{1}{2}\Re(z) e^t} dt=\frac{2e^{\frac{1}{2}e \Re(z)}}{\pi} \int_{\frac{e}{2}}^\infty \frac{e^{-\Re(z) u}}{u} du=\frac{2e^{\frac{1}{2} e \Re(z)}}{\pi} \int_{\frac{1}{2}e \Re(z)}^\infty \frac{e^{-u}}{u} du\\
   &=& \frac{2}{\pi} \int_{0}^\infty \frac{e^{-u}}{u+\frac{1}{2} e \Re(z)} = \int_{1}^\infty(\cdots)+\int_{0}^1(\cdots) \leq C+ C\left|\log\left(\frac{1}{2} e\Re(z)\right)\right| \\
   &\leq& C (1+ |\log (\Re(z))|).
     \eee
Above, we used that $\log(1+\frac{1}{x})\leq 1+ |\log(x)|$ for $x>0$. This ends the proof.

   \subsection{Proof of Theorem \ref{th:HF}}

We study first the high frequency limit in the interior of the domains $\{z>0\}$, $\{z \in (-L,0)\}$ and $\{z<-L\}$, and only detail the most interesting case $\{z \in (-L,0)\}$. For this, let $\varphi$ be a smooth test function with support in $\{z \in (-L,0)\}$, and let $v_\varphi= v \varphi$. Then $v_\varphi$ verifies
 \begin{align} \label{eqvp}
   \eta^2 \Delta v_\varphi(x)+ &\big(k^2(x)+i \eta \mu(x)\big)  v_\varphi(x)\\ \nonumber
&= (2\pi)^2 \varphi(x)\chi(x) u(x) dW_x+\eta^2 v(x) \Delta \varphi(x)+2 \eta^2 \nabla \varphi(x) \cdot \nabla v(x)\\
   &:=s_1(x)+\eta^2 s_2(x).\nonumber
 \end{align}
Note, since $\varphi$ is supported in $\{z \in (-L,0)\}$, that we have above $k(x)=k_e$ and $\mu(x)=\mu_e$. The starting point is the following expression of Wigner transform of $v_\varphi$ in the Fourier space:
\be \label{Fwig}
\widehat{\calW[v_\varphi]}(q,p)=\frac{1}{(2\pi)^d \eta^d} \hat v_\varphi\left(\frac{q}{2}+\frac{p}{\eta} \right)\hat v^*_\varphi\left(\frac{q}{2}-\frac{p}{\eta} \right).
\ee
Taking then the Fourier transform of \fref{eqvp} yields

 $$
 \widehat{v}_\varphi (\xi)=\frac{\widehat{s}_1(\xi)+\eta^2 \widehat{s}_2(\xi)}{a_\eta(\xi)}, \qquad a_\eta(\xi)=k_e^2+i \eta \mu_e-\eta^2 |\xi|^2.
 $$
Since $\E\{\hat s_1\}=0$, we find
\be \label{Ewig}
\E\{ \widehat{v}_\varphi (\xi) \widehat{v}^*_\varphi (\xi') \} =\frac{\E\{ \hat s_1(\xi) \hat s_1^*(\xi')\}+ \eta^4 \hat s_2(\xi) \hat s_2^*(\xi')}{a_\eta(\xi)a_\eta^*(\xi')}.
 \ee
Moreover, \fref{varWien} implies that
 \be \label{Es1}
 \E\{ \hat s_1(\xi) \hat s_1^*(\xi')\}=(2 \pi)^4 \tau^2 \int_{\Rm^d}e^{-i (\xi+\xi') \cdot x} |\varphi \chi u|^2(x) dx= (2 \pi)^4 \tau^2 \widehat{|\varphi \chi u|^2} (\xi+\xi').
 \ee
We next realize that
 \begin{align*}
   &\frac{1}{a_\eta(\frac{q}{2}+\frac{p}{\eta})a_\eta^*(\frac{q}{2}-\frac{p}{\eta})}=\frac{1}{2 \eta(i  \mu_e-p \cdot q)} \\
   &\hspace{2cm}\times \left(\frac{1}{-i \eta \mu_e+ k_e^2-|p|^2+\eta p \cdot q-\eta^2|q|^2/4}
 -\frac{1}{i \eta \mu_e+ k_e^2-|p|^2-\eta p \cdot q-\eta^2|q|^2/4}
 \right)
 \end{align*}
and since 
 $$
 \lim_{\eta \to 0} \left(\frac{1}{x+ i \eta}-\frac{1}{x-i \eta}\right)=- 2i \pi \delta(x),
 $$
we find, in the distribution sense,
$$
\lim_{\eta \to 0}\frac{\eta }{a_\eta(\frac{q}{2}+\frac{p}{\eta})a_\eta^*(\frac{q}{2}-\frac{p}{\eta})}=\frac{  \pi}{ (\mu_e+i p \cdot q)}\delta(|p|^2-k_e^2).
$$
Using this, together with \fref{Fwig}, \fref{Ewig} and \fref{Es1} finally yields
$$
( \mu_e+ i p \cdot q)\calW[v_\varphi](q,p)=\pi \delta(|p|^2-k_e^2) \left( \frac{(2 \pi)^4\tau^2}{(2\pi)^d \eta^{d+1}}\widehat{|\varphi \chi u|^2} (q)+ \eta^3 \calW[s_2](q,p)\right)+o(\eta).
$$
The term $\eta^3 \calW[s_2](q,p)$ leads to a negligible contribution in the interior of the domain $\{z \in (-L,0)\}$, and can therefore be neglected. After an inverse Fourier transform, we then obtain for $\calW[v_\varphi]$:
$$
( \mu_e+ p \cdot \nabla_ x)\calW[v_\varphi](x,p)=\frac{\pi (2 \pi)^4\tau^2}{\eta^{d+1}}\delta(|p|^2-k_e^2) |\varphi \chi u|^2(x) (q)+o(\eta),
$$
where $o(\eta)$ denotes a term that converges to zero as $\eta \to 0$ in the distribution sense. Moreover, we have $\calW[v_\varphi] =|\varphi|^2 \calW[v] +o(\eta)$ as $\eta \to 0$, and we recover the equation in the theorem after removing the test function $\varphi$. The calculations are similar in the domains $\{z>0\}$ and $\{z<-L\}$, and lead to a free transport equation of the form $( \mu(x)+ p \cdot \nabla_ x)\calW[v](x,p)=o(\eta)$.

The reflection-transmission boundary conditions are obtained for instance by adapting the results of \cite{CPDE-half} to our layered case. This ends the proof.

\section{Conclusion} \label{conc} We have derived in this work, under appropriate conditions on the wavelength, the typical length scale of the fluctuations and their variances, an asymptotic model for the first-order corrector to the homogenization limit. We have considered a simplified model where the inhomogeneities occupy a single layer with flat interfaces, and where the propagation of waves is described by a random Helmholtz equation. The corrector takes the simple form of a Gaussian field whose statistics depends on the homogenized solution and the homogenized Green's function. In addition, when the thickness of layer is sufficiently large compared to the wavelength, we have obtained a transport model for the correlations of the correctors. Finally, we have explained how to generalize our results to more complicated settings of multi-layers with rough boundaries.

This work is the first step towards the understanding of the true physical problem and it leaves many questions open. We describe next some of those problems. Firstly, what is the form of the corrector when the fluctuations of the media have stronger strength? In this case, the homogenized model involves the resolution of a cell problem and the characterization of the corrector becomes much more difficult. Secondly, how significant are boundary effects? It is unclear whether the main contribution to the backscattered signal comes from volumic scattering by the heterogeneities in the sea ice or from surface scattering from rough interfaces, in particular the sea ice / sea water interface that is the most relevant for our purpose. Finally, how all of this translates to the true physical situation that involves electromagnetic waves and Maxwell's equations? These questions will be addressed in future works.
\bibliographystyle{siam}
\bibliography{bibliography.bib}

\end{document}